\documentclass[11pt]{article}
\usepackage{latexsym}
\usepackage{indentfirst}
\usepackage{enumerate}
\usepackage{cite}
\usepackage{amssymb}
\usepackage{amsthm}
\usepackage{amsmath}
\usepackage{amsfonts}
\usepackage{txfonts}
\usepackage{graphicx}
\usepackage{hyperref}
\usepackage{color}

\usepackage[T1]{fontenc}
\usepackage[latin2]{inputenc}
\title{The Cucker-Smale equation: singular communication weight, measure-valued solutions and weak-atomic uniqueness}

\author{Piotr B. Mucha\footnote{{\tt p.mucha@mimuw.edu.pl}} \quad\quad
Jan Peszek\footnote{{\tt j.peszek@mimuw.edu.pl}}\ \footnote{JP was supported by International Ph.D. Projects Programme of Foundation for Polish Science operated within the Innovative Economy Operational Programme 2007-2013 funded by 
the European Regional Development Fund (Ph.D. Programme: Mathematical Methods in Natural Sciences) and partially supported by the Polish NCN grant  PRELUDIUM 2013/09/N/ST1/04113.}\\
.\\
{\it Institute of Applied Mathematics and Mechanics,}\\
{\it University of Warsaw, ul. Banacha 2,}\\
{\it 02-097 Warsaw, Poland}
}

\date{\today}

\renewcommand{\it}{\sl}

plus 6pt minus 12pt
\newcommand{\barint}{
         \rule[.036in]{.12in}{.009in}\kern-.16in
          \displaystyle\int  }
\def\r{{\mathbb{R}}}

\begin{document}

\newtheorem{theo}{\bf Theorem}[section]
\newtheorem{coro}{\bf Corollary}[section]
\newtheorem{lem}{\bf Lemma}[section]
\newtheorem{rem}{\bf Remark}[section]
\newtheorem{defi}{\bf Definition}[section]
\newtheorem{ex}{\bf Example}[section]
\newtheorem{fact}{\bf Fact}[section]
\newtheorem{prop}{\bf Proposition}[section]
\newtheorem{prob}{\bf Problem}[section]

%\makeatletter \@addtoreset{equation}{section}
%\renewcommand{\theequation}{\thesection.\arabic{equation}}
%\makeatother

\newcommand{\ds}{\displaystyle}
\newcommand{\ts}{\textstyle}
\newcommand{\ol}{\overline}
\newcommand{\wt}{\widetilde}
\newcommand{\ck}{{\cal K}}
\newcommand{\ve}{\varepsilon}
\newcommand{\vp}{\varphi}
\newcommand{\pa}{\partial}
\newcommand{\rp}{\mathbb{R}_+}
\newcommand{\hh}{\tilde{h}}
\newcommand{\HH}{\tilde{H}}
\newcommand{\cp}{{\rm cap}^+_M}
\newcommand{\hes}{\nabla^{(2)}}
\newcommand{\nn}{{\cal N}}
\newcommand{\dix}{\nabla_x\cdot}
\newcommand{\dv}{{\rm div}_v}
\newcommand{\di}{{\rm div}}
\newcommand{\pxi}{\partial_{x_i}}
\newcommand{\pmi}{\partial_{m_i}}
\newcommand{\tor}{\mathbb{T}}
\newcommand{\pot}{\mathcal{v}}

\maketitle

\begin{abstract}
The Cucker-Smale flocking model belongs to a wide class of kinetic models that describe a collective motion of interacting particles that exhibit some specific tendency e.g. to aggregate, flock or disperse. 
 The paper examines the kinetic Cucker-Smale equation with a singular communication weight. Given a compactly supported measure as an initial datum we construct a global in time  weak measure-valued solution in 
the space $C_{weak}(0,\infty;\mathcal{M})$. The solution is defined as a mean-field limit of the empirical distributions of particles, which dynamics is governed by the Cucker-Smale particle system. The
 studied communication weight is $\psi(s)=|s|^{-\alpha}$ with $\alpha \in (0,\frac 12)$. This range of singularity
admits sticking of characteristics/trajectories.
The second result concerns the weak--atomic uniqueness property stating that a weak solution initiated by a finite sum of atoms, i.e. Dirac deltas in the form $m_i \delta_{x_i} \otimes
\delta_{v_i}$, preserves its atomic structure. Hence they coincide with unique solutions to the system of ODEs associated with the Cucker-Smale particle system.
\end{abstract}
%{Cucker-Smale model \and singular potential \and alignment models \and mean-field limit \and measure-valued solutions}
%\subclass{82C22 \and 82C40 \and 92C17}

%\bigskip

%{\bf Conflict of Interest:} The authors declare that they have no conflict of interest.

%\newpage

\section{Introduction}\label{intro}
Flocking, swarming, aggregation - there is a multitude of actual real-life phenomena that from the mathematical point of view can be interpreted as one of these concepts. The mathematical description
 of collective dynamics of self-propelled agents with nonlocal interaction originates from one of the basic equations of the kinetic theory -- Vlasov's equation from 1938. Recently it was noted that such 
models provide a way to describe a wide range of phenomena that involve interacting agents with a tendency to aggregate their certain qualities. 
This approach proved to be useful and the language of aggregation now appears not only in the models of groups of animals but also in the description of seemingly unrelated phenomena such as the emergence 
of common languages in primitive societies, distribution of goods or reaching a consensus among individuals \cite{lang, cons, cons2, goods}.
The main class of kinetic equations associated with aggregation models reads as follows:
\begin{align}\label{k}
\partial_t f + v\cdot\nabla f + \dv[(k*f)f]=0,
\end{align}
where $f=f(x,v,t)$ is usually interpreted as the density of those particles that at time $t$ have position $x$ and velocity $v$. The function $k$ is the kernel of the potential governing the motion of particles.
 It is responsible for the non-local interaction between particles and depending on it the particles may exhibit various tendencies like to flock, aggregate or to disperse. 
The common properties of  kernels $k$ required in most models include Lipschitz continuity and boundedness.
In such case the particle system associated with (\ref{k}) is well posed, the characteristic method can be performed for (\ref{k}) and one can  pass from the particle system 
to the kinetic equation by the mean-field limit procedure. Our goal is to consider singular $k$ that is neither Lipschitz continuous nor bounded and refine the mean-field limit to be applicable in such scenario. 
Singularity of the kernel admits possibility of sticking of characteristics, which causes a loss of backward uniqueness in time for (\ref{k}).
The present paper studies the case of the  Cucker-Smale (CS) flocking model.

In \cite{cuc1} from 2007, Cucker and Smale introduced a model for the flocking of birds associated with the following system of ODEs:
\begin{align}\label{cs}
\left\{
\begin{array}{ccl}
\displaystyle\frac{d}{dt}x_i &=& v_i,\\
\displaystyle\frac{d}{dt}v_i &=& \displaystyle\sum_{j=1}^Nm_j(v_j-v_i)\psi(|x_j-x_i|),
\end{array}
\right.
\end{align}
where $N$ is the number of the particles while $x_i(t)$, $v_i(t)$ and $m_i$ denote the position and velocity of $i$th particle at time $t$ and its mass, respectively. 
Function $\psi:[0,\infty)\to[0,\infty)$ usually referred to as {\it the communication weight} is nonnegative and nonincreasing and can be vaguely interpreted as the perception of particles.
The communication weight plays the crucial role in our investigations. It characterizes mathematical difficulties and determines possible physical interpretations of the studied model.

As $N\to\infty$ the particle system is replaced by the following Vlasov-type equation:
\begin{align}\label{cscont}
\partial_tf+v\cdot\nabla f+{\rm div}_v[F(f)f]=0,\ \ x\in\r^d,\ v\in\r^d,\\
F(f)(x,v,t):=\int_{\r^{2d}}\psi(|y-x|)(w-v)f(y,w,t)dwdy\nonumber,
\end{align}
which can be written as (\ref{k}) with $k(x,v)=v\psi(|x|)$. As mentioned before we are considering (\ref{cscont}) with a singular kernel
\begin{align}\label{psi}
\psi(s)=
\left\{
\begin{array}{ccc}
s^{-\alpha}&{\rm for}& s>0,\\
\infty&{\rm for}& s=0,
\end{array}
\ \ \ \ \ \ \alpha>0.\right.
\end{align}
Before we proceed with a more detailed statement of our goals let us briefly introduce the current state of the art for models of flocking and the motivations behind studying such models with singular kernels. 
The literature on aggregation models associated with Vlasov-type equations of the form (\ref{k}) is rich thus we  mention only a few examples  of the most popular branches of the field. 
Here we find  analysis of time asymptotics (see e.g. \cite{hakalaru}) and pattern formation (see e.g. \cite{hajekaka, top}) or analysis of the models with additional forces that simulate various 
natural factors (see e.g. \cite{car3, dua1} - deterministic forces or \cite{cuc4} - stochastic forces). 
The other variations of the model include forcing particles to avoid collisions (see e.g. \cite{cuc2}) or to aggregate under the leadership of certain individuals (see e.g. \cite{cuc3}). A well rounded analysis 
of a model that includes effects of attraction, repulsion and alignment is presented in \cite{can}. The story of the CS model should probably begin with \cite{vic} by Vicsek {\it et al.},
 where a model of flocking with nonlocal interactions was introduced and it is widely recognized to be up to some degree an inspiration for \cite{cuc1}. Since 2007 the CS model with a regular communication weight
 of the form
\begin{align}\label{cucu}
\psi_{cs}(s)=\frac{K}{(1+s^2)^\frac{\beta}{2}}, \ \ \ \beta\geq 0,\ \ \  K>0
\end{align}
was extensively studied in the directions similar to those of more general aggregation models (i.e. collision avoiding, flocking under leadership, asymptotics and pattern formation as well as additional 
deterministic
 or stochastic forces - see \cite{aha2, halele, park, shen, hahaki, car}). Particularly interesting from our point of view is the case of passage from the particle system (\ref{cs}) to the kinetic equation 
(\ref{cscont}), which in case of the regular communication weight was done for example in \cite{haliu} or \cite{hatad}. For a more general overview of the passage from microscopic to mesoscopic and macroscopic 
descriptions in aggregation models of the form (\ref{k}) we refer to \cite{rec, deg1, deg3}.

In the paper \cite{haliu} from 2009 the authors considered the CS model with the singular weight (\ref{psi}) obtaining asymptotics for the particle system but even the basic question of existence of solutions
 remained open till later years. It turned out that system (\ref{cs}) possesses drastically different qualitative properties depending on whether $\alpha\in(0,1)$ or $\alpha\in[1,\infty)$. More precisely in
 \cite{ahn1} the authors observed that for $\alpha\geq 1$ the trajectories of the particles exhibit a tendency to avoid collisions, which they used to prove conditional existence and uniqueness of smooth
 solutions to the particle system. On the other hand in \cite{jpe} the author proved existence of so called {\it piecewise weak} solutions to the particle system with $\alpha\in(0,1)$ and gave an example of
 solution that
 experienced not only collisions of the trajectories but also sticking (i.e. two different trajectories could start to coincide at some point). This dichotomy is an effect of integrability (or of the lack 
of thereof) of $\psi$ in a neighborhood of $0$. 
It is also the reason why the approach to the CS model should vary depending on $\alpha$. One of the latest contributions to this topic is \cite{carchoha} where the authors showed local in-time well posedness
 for the kinetic equation (\ref{cscont}) with a singular communication weight (\ref{psi}) and with an optional nonlinear dependence on the velocity in the definition of $F(f)$. They also presented a thorough 
analysis of the asymptotics for this model. The other more recent addition is \cite{jps}, where the author proved existence and uniqueness of $W^{1,1}$ strong solutions to the particle system (\ref{cs}) with
a singular weight (\ref{psi}) and $\alpha\in(0,\frac{1}{2})$. 

Our present results are strongly dependent on regularity of solutions to the particle system, hence we assume, after \cite{jps}, that throughout the paper $\alpha\in(0,\frac{1}{2})$. However, let us mention that these results (in particular Proposition \ref{pop}) can be generalized to models of flocking with nonlinear dependence on the velocity in the alignment force term $F(f)$. Such models are considered for example in \cite{carchoha}, where the force term takes the form
\begin{align*}
F_{\beta}(f):=\int_{\r^{2d}}\psi(|x-y|)\nabla_v\phi(v-w)f(y,w,t)dydw
\end{align*}
with $\phi(s)\sim s^\beta$ for $\beta>\frac{3-d}{2}$. In particular $\phi$ with large $\beta$ reduces the impact of the singularity of $\psi$. This approach allows to push the singularity of the communication weight $\psi$ up to $1$. However, this interesting extension is outside of the scope of this paper.

As a final remark we suggest a possible application of this type of mathematical models. The phenomenon of sticking of trajectories gives a possibility of creation of Dirac measure solutions from regular 
distributions. Like in the case of \cite{cafig} for the equations of attraction/repulsion. Here we think about formation of polymers from a solution of monomers. Still this qualitative nature of the CS model 
is an open question. However it would give a nice description of polymerization. Due to the possibility of sticking of trajectories the kinetic structure of the considered model is more complex than in \cite{jahu} 
where the authors considered singular kinetic models that preserved $L^\infty$ norm of the distribution of the particles.

In the model with regular weight its purpose is to suppress the distant interactions between particles. However from the modeling point of view it is often convenient to also amplify the local interactions,
 which was done for example in \cite{mo} by introducing a different nonsymmetric CS-type model known as the Mosch-Tadmor model. Singular communication weight in the CS model can also be viewed as a
 less effective yet easier to analyze way to emphasize the local interactions between particles.

%{\bf tu jeszcze bbgky?}

\subsection{Main goal - the CS model with a singular communication weight}
We aim at solving the issue of  well posedness for (\ref{cscont}) with the singular weight (\ref{psi})
and initial data from the class of Radon measures. The goal is twofold:

-- Prove existence and analyze continuous dependence on the initial data. 
The existence is obtained by approximating  measure  solutions to (\ref{cscont}) by solutions to  particle system (\ref{cs}) 
using the mean-field limit, similarly to \cite{haliu}. The key obstacle is a lack of sufficient information about the continuous dependence
for solutions to  particle system (\ref{cs}), thus we are not allowed to apply the standard approach.
To our best knowledge the most that can be assumed is solvability of (\ref{cs})  in the  $W^{1,1}$- class that has been proved in \cite{jps}.
 Therefore  by results from \cite{jps} we restrict our considerations to $\alpha\in(0,\frac{1}{2})$ and modify the mean-field limit 
procedure to that regularity.

-- Prove the  weak-atomic uniqueness property to system (\ref{cscont}). 
It means that any weak solution is unique and corresponds to a solution to the particle system (\ref{cs}) 
provided  it initiates from a finite sum of Dirac's deltas $m_i \delta_{x_i(t)}\otimes \delta_{v_i(t)}$.
 Thus any atomic solution is preserved by kinetic equation (\ref{cscont}), and since it is generated by particle system (\ref{cs}), it is unique. This result is a step in the direction of stability of solutions to \eqref{cscont}. We elaborate further on the difficulties in obtaining stability in Remark \ref{refrem}.

The paper is organized as follows. In section \ref{sec2} we provide  the preliminary definitions and tools required throughout the paper,
in particular we introduce  the weak formulation for (\ref{cscont}).
 In section \ref{sec3} we state the main result along with the overview of the proof. In section \ref{sec4} we present the proof of the existence part, while in section \ref{sec5} we establish 
the weak-atomic uniqueness of the solutions. The paper is closed with Appendix \ref{app} where one can find more  technical/tedious elements of proofs and a simple proof of uniqueness to particle system (\ref{cs}).

\section{Preliminaries and notation}\label{sec2}
In this section we present the basic toolset and  definitions
 of  considered problems. 
 Let $\Omega\subset \r^d$ be an arbitrary domain with $d\in\mathbb{N}$. By $W^{k,p}(\Omega)$ we denote the Sobolev's space of the functions with up to $k$th weak derivative belonging to space $L^p(\Omega)$,
 while by $C(\Omega)$ and $C^1(\Omega)$ we denote the space of continuous and continuously differentiable functions, respectively. Hereinafter, $B((x_0,v_0),R)=B_{x,v}((x_0,v_0),R)$ denotes a ball in $\r^{2d}$
 centered in $(x_0,v_0)$ 
with radius $R$. On the other hand $B_x(x_0,R)$ and $B_v(v_0,R)$ denote balls in $\r^d$ with radius $R$ centered in $x_0$ and $v_0$, respectively. For any positive $a$, by $aB_v(v_0,R)$ we understand a homothetic 
transformation of $B_v(v_0,R)$, i.e., $B_v(v_0a,Ra)$.
Throughout the paper letter $C$ denotes a generic positive constant that may change in the same inequality and usually depends on other constants that are of less importance from the point of view of estimates.

\subsection{Bounded-Lipschitz distance}
The standard tool used in the studies of Vlasov-type models are Wasserstein metrics. They provide  convenient topologies on the space of Radon measures and are often used in the research on CS model. 
The general definition of Wasserstein distances is relatively complex, however Wasserstein-1 (or Kantorovich-Rubinstein) distance is, in the sense 
explained 
in \cite{wasser} p. 26, equivalent to the easily defined bounded-Lipschitz distance. Due to such convenient representation we use only the bounded-Lipschitz distance but it 
is worth to note that in similar cases (e.g. \cite{carchoha}) other Wasserstein metrics can be applied.
\begin{defi}[Bounded-Lipschitz distance]\label{bld}
For any probabilistic measures $\mu$ and $\nu$ we define
\begin{align*}
d(\mu,\nu):=\sup_g\left|\int_\Omega gd\mu -\int_\Omega gd\nu\right|,
\end{align*}
where the supremum is taken over all bounded and Lipschitz continuous functions $g$, such that $\|g\|_\infty\leq 1$ and $Lip(g)\leq 1$.
\end{defi}
In the above definition $\|g\|_\infty$ and $Lip(g)$ represent the $L^\infty$ norm and Lipschitz constant of $g$. We also need to 
distinct between spaces of measures with different topologies i.e. we denote ${\mathcal M}={\mathcal M}(\Omega)=({\mathcal M}, TV)$ as the space of finite Radon measures defined on $\Omega$ with 
{\it total variation} topology and we denote
$({\mathcal M},d)$ as the space of finite Radon measures defined on $\Omega$ with bounded-Lipschitz distance topology. The importance of the space $({\mathcal M},d)$  comes from the prime difference 
between the bounded-Lipschitz distance and the total variation. Namely, for $x_1\neq x_2$, 
$TV(\delta_{x_1}-\delta_{x_2}) = 2$ ,
while
$d(\delta_{x_1},\delta_{x_2})\leq 2|x_1-x_2|$.
In particular,  if $x_n\to x$ in $\Omega$ then $\delta_{x_n}\to\delta_x$ in $d$, which is not the case in $TV$.

In our considerations a crucial role is played by
\begin{align*}
{\mathcal M}_+ := \Big\{\mu\in{\mathcal M}:\  \mu\ {\rm is\ nonnegative}\Big\}
\end{align*}
both with $TV$ and $d(\cdot,\cdot)$ topology. If $\Omega$ is a compact subset of $\r^d$, then ${\mathcal M}$ is isomorphic to $(C_b(\Omega))^*$\footnote{This nice property does not hold in general. 
For example if $\Omega=\r^d$ then $(C_b(\Omega))^*$ is isomorphic to the space of regular bounded finitely additive measures, while $(C_0(\Omega))^*$ is isomorphic to ${\mathcal M}$.}. There is a convenient
 relation between the weak * topology in $(C_b(\Omega))^*$ and the topology generated by the bounded-Lipschitz distance on ${\mathcal M}$ that we present below.

\begin{defi}
We say that a sequence $\{\mu_n\}_{n\in{\mathbb N}}\subset {\mathcal M}$ is tight if for all $\epsilon>0$ there exists a compact $K_\epsilon\subset\subset\Omega$ such that for all $n\in{\mathbb N}$ we have
\begin{align*}
|\mu_n|(\Omega\setminus K_\epsilon)<\epsilon,
\end{align*}
where $|\mu|$ is the total variation measure of $\mu$.
\end{defi}

\begin{prop}\label{narrow}
Suppose that $\Omega$ is compact and let $\{\mu_n\}_{n\in{\mathbb
 N}}$ be a tight sequence in ${\mathcal M}$ and let $\mu\in{\mathcal M}$. Then $\mu_n\to\mu$ weakly * if and only if $d(\mu_n,\mu)\to 0$ and $\{\mu_n\}_{n\in{\mathbb N}}$ is bounded in ${\mathcal M}$.
\end{prop}
\begin{proof}
The proof of Proposition \ref{narrow} is a modification of the proof of Theorem 2.7 from \cite{star}.
\end{proof}

\begin{rem}\rm
Since in our considerations $\Omega$ is compact, our situation is more straightforward than in a general case (e.g. in \cite{star}). In such case the so called weak convergence of measures 
(otherwise known as the narrow convergence) is equivalent to the weak * convergence in $(C_b(\Omega))^*$.
\end{rem}

In our considerations we deal only with nonnegative measures, which equipped with the bounded-Lipschitz distance are a complete metric space.

\begin{prop}\label{compl}
The space $({\mathcal M}_+,d)$ is a complete metric space.
\end{prop}

The following corollary is the very reason for which the bounded-Lipschitz distance is applied. It serves us as a topology with pointwise sequential compactness for measure-valued functions.
 What we mean is that if $f_n:[0,T]\mapsto ({\mathcal M},d)$ and $f_n$ are uniformly bounded in $L^\infty(0,T;({\mathcal M},TV))$ then for each $t\in[0,T]$ the sequence $f_n(t)$ is relatively 
compact in $({\mathcal M},d)$, which is one of the assumptions of the Arzela-Ascoli theorem.
\begin{coro}\label{comp}
Let $\{\mu_n\}_{n\in{\mathbb N}}$ be a sequence bounded in $({\mathcal M}_+,TV)$ with supports contained in some given ball. Then there exists a $({\mathcal M}_+,d)$-convergent subsequence $\{\mu_{n_k}\}$.
\end{coro}
\begin{proof}
Since the supports of $\mu_n$ are uniformly bounded then there exist compact set $K\subset\subset\Omega\subset\subset\r^d$ such that
\begin{align*}
\bigcup_{n}{\rm supp}\,\mu_n\subset K.
\end{align*}
Thus we may treat $\{\mu_n\}_{n\in{\mathbb N}}$ as a sequence of measures defined on a compact $\Omega$. Then $({\mathcal M},TV)$ is isomorphic to $(C_b(\Omega))^*$, which is a separable normed vector 
space then by Banach-Alaoglu theorem the set $\{\mu_n:n=1,2,...\}$ is sequentially weakly * compact in $(C_b(\Omega))^*$. Therefore there exists a measure $\mu\in(C_b(\Omega))^*$ such that up to a 
subsequence $\mu_n$ converges to $\mu$ weakly * in $(C_b(\Omega))^*$. Proposition \ref{narrow} implies that the weak * $(C_b(\Omega))^*$ convergence is equivalent to the convergence in 
$d(\cdot,\cdot)$ if only $\{\mu_n\}_{n\in{\mathbb N}}$ is tight (which is true since $\mu_n$ vanish on $\Omega\setminus K$). Thus $\mu_n$ converges to $\mu$ in $d(\cdot,\cdot)$. Finally since 
by Proposition \ref{compl} $({\mathcal M}_+,d)$ is a complete space we conclude that actually $\mu\in({\mathcal M}_+,d)$ and the proof is finished.
\end{proof}

Lastly we present a useful lemma related to the bounded-Lipschitz distance.
\begin{lem}\label{bldist}
Let $d(\cdot,\cdot)$ be the bounded-Lipschitz distance. Then for any $\mu,\nu\in{\mathcal M}$ and any bounded and Lipschitz continuous function $g$, we have
\begin{align*}
\left|\int_\Omega gd\mu - \int_\Omega gd\nu\right|\leq \max\{\|g\|_\infty,Lip(g)\} d(\mu,\nu).
\end{align*}
\end{lem}
\begin{proof}
The proof of this lemma belongs to the standard theory and can be found, for example, in \cite{haliu}.
\end{proof}

\subsection{Measure-valued solutions to the kinetic equation}

We introduce the following weak formulation for (\ref{cscont}):
\begin{defi}\label{weakdef}
Let $T>0$. We  say that $f$ is a weak solution to  (\ref{cscont}) with the initial data $f_0\in {\mathcal M}_+$, 
such that  ${\rm supp}f_0 \subset B({\mathcal R_0})$ with ${\mathcal R_0} >0$ if 
\begin{enumerate}
\item  $f\in L^\infty(0,T;{\mathcal M}_+)$ and $\partial_t f\in L^p(0,T;(C^1_b(B({\mathcal R})))^*)$ for some $p>1$;
\item 
% \begin{align*}
${\rm supp}f(t) \subset B({\mathcal R})$ for $t\in (0,T]$
% \end{align*}
for some positive constant ${\mathcal R}$;
\item The following identity holds:
\begin{align}\label{weak}
\int_0^T\int_{\r^{2d}}f[\partial_t\phi + v\nabla\phi]dxdvdt + \int_0^T\int_{\r^{2d}}F(f)f\nabla_v\phi dxdvdt = \\\nonumber
=-\int_{\r^{2d}}f_{0}\phi(\cdot,\cdot,0)dxdv
\end{align}
for all $\phi\in {\mathcal G}$, where
\begin{align*}
{\mathcal G}:=\Big\{&\phi\in C^1([0,T)\times\r^{2d}): \partial_t\phi,\nabla\phi,\nabla_v\phi\ {\rm are\ bounded}\\ &{\rm and\ Lipschitz\ continuous}\ {\rm and}\ \phi\ {\rm has\ a\ compact\ support\ in}\ t\Big\};
\end{align*}
\item The function $g(x,y,v,w,t):=(w-v)\psi(|x-y|)$ is integrable with respect to the measure $f(x,v,t)\otimes f(y,w,t)dxdvdydw$, 
i.e.  term $F(f)$ is defined as a measure with respect to the measure $fdxdv$. In particular by Fubini's theorem the integral
\begin{align*}
\int_0^T\int_{\r^{2d}}F(f)f\nabla_v\phi dxdvdt = \int_0^T\int_{\r^{4d}}g\nabla_v\phi f\otimes f dxdvdydwdt
\end{align*}
is bounded and the term $\dv [F(f)f]$ is well defined as a distribution;
\item For each pair of concentric balls $B((x_0,v_0),r)\subset B((x_0,v_0),R) \subset \r^{2d}$, the following statement holds: if
\begin{align}
{\rm supp}f_0\cap B((x_0,v_0),R)\subset B((x_0,v_0),r), \label{ccentr}
\end{align}
then there exists $T^*\in[0,T]$, such that
\begin{align}\label{propag}
{\rm supp}f(t)\cap B\left((x_0,v_0),\frac{3R+r}{4}\right) \subset B\left((x_0,v_0),\frac{r+R}{2}\right)
\end{align}
for all $t\in[0,T^*]$.
\end{enumerate}
\end{defi}
\begin{rem}\label{remcor}\rm
There is a natural question of the correspondence between solutions to (\ref{cscont}) in the sense of Definition \ref{weakdef} and solutions to (\ref{cs}).
 The answer to this question is to some merit positive, which we explain below.
Let
\begin{align}\label{discini}
f_0(x,v):=\sum_{i=1}^Nm_i\delta_{x_{i,0}}(x)\otimes\delta_{v_{i,0}}(v)
\end{align}
with $\sum_{i=1}^Nm_i=1$. Then $f_0$ defines an initial data $x_0=(x_{1,0},...,x_{N,0})$, $v_0=(v_{1,0},...,v_{N,0})$ for the system of ODE's (\ref{cs}). For this system let $(x,v)$ be a sufficiently
 smooth\footnote{By ''sufficiently smooth'' we mean for instance that $(x,v)\in W^{1,1}([0,T])$, which is a reasonable assumption in view of Theorem \ref{jpsmain}.} solution. Then the function
\begin{align}\label{remeq}
f(x,v,t):=\sum_{i=1}^Nm_i\delta_{x_{i}(t)}(x)\otimes\delta_{v_{i}(t)}(v)
\end{align} 
is a solution of (\ref{cscont}) in the sense of Definition \ref{weakdef} with the initial data $f_0$. Indeed, if we plug $f$ defined in (\ref{remeq}) into (\ref{weak}), by a simple use of a chain rule, we obtain
\begin{align*}
\int_0^T\sum_{i=1}^Nm_i\Big((\partial_t\phi)(x_i,v_i,t) + v_i(\nabla\phi)(x_i,v_i,t)\Big)\\
 + \sum_{i,j=1}^Nm_im_j\psi(|x_i-x_j|)(v_j-v_i)(\nabla_v\phi)(x_i,v_i,t)dt \\
= \int_0^T\sum_{i=1}^Nm_i\frac{d}{dt} \phi(x_i(t),v_i(t),t)dt = -\sum_{i=1}^Nm_i\phi(x_{i,0},v_{i,0},t) = -\int_{\r^{2d}}f_{0}\phi(\cdot,\cdot,0)dxdv
\end{align*}
for all $\phi\in{\mathcal G}$.
The converse assertion that a solution to (\ref{cscont}) in the sense of Definition \ref{weakdef} corresponds to a solution of (\ref{cs}) is also true provided that the initial data are of the form (\ref{discini}).
 However, the proof is much more involved and it is in fact the second part of the paper.
\end{rem}
\begin{defi} We say that $f$ is an atomic solution if it is of form (\ref{remeq}).
\end{defi}

\begin{defi}
In case of solutions of particle system (\ref{cs}) we say that $i$th and $j$th particles {\it collide} at the time $t$ if and only if $x_i(t)=x_j(t)$ and we say that they {\it stick together} at the time $t$ if 
and only if $x_i(t)=x_j(t)$ and $v_i(t)=v_j(t)$.
\end{defi}

\begin{rem}\rm\label{N-N}
Throughout the paper whenever we consider a solution of system (\ref{cs}) we assume that the number of particles is constant in time. However if in our proofs (particularly in Section \ref{sec5}) any 
two particles stick together, then we tend to treat them as a single particle with a bigger mass, thus reducing the total number of particles in the system. This is justified by the fact that according to \cite{jps} solutions to (\ref{cs}) with $\alpha\in(0,\frac{1}{2})$ are unique (see Theorem \ref{jpsmain}).
\end{rem}

\begin{rem}\rm
Point 5 of Definition \ref{weakdef} requires some explanation. Its purpose is to establish a local control over the propagation of the support of $f$. Basically if we can divide the support of $f_0$ into 
two parts of distance $R-r$, then in some small time interval $[0,T^*]$ the distances between those parts is no lesser than $\frac{R-r}{4}$.
\end{rem}

\begin{rem}\rm
In Section \ref{sec5} we frequently test our weak solution by various test functions that at the first glance may seem not admissible. 
In particular we test with functions with derivatives in $x$ and $v$ not necessarily Lipschitz continuous. This is however correct since by a simple density argument we may test (\ref{weak}) by $C^1$ functions. 
Moreover we are allowed to test (\ref{weak}) by functions that are not compactly supported in time. In such case we get a version of (\ref{weak}) with both endpoints of the time interval, 
\begin{align}\label{ugh}
\int_0^T\int_{\r^{2d}}f[\partial_t\phi + v\nabla\phi]dxdvdt + \int_0^T\int_{\r^{2d}}F(f)f\nabla_v\phi dxdvdt = \\=\int_{\r^{2d}}f(t)\phi(\cdot,\cdot,t)dxdv-\int_{\r^{2d}}f_{0}\phi(\cdot,\cdot,0)dxdv.\nonumber
\end{align}
The justification of the above equation is standard and can be found in the proof of Proposition \ref{pop},$(v)$ in Appendix \ref{app}.
\end{rem}

\section{Main result}\label{sec3}

The main result of the paper is the following.
\begin{theo}\label{main3}
Let $0<\alpha<\frac{1}{2}$. For any compactly supported initial data $f_0\in {\mathcal M}$ and any $T>0$, Cucker-Smale's flocking model (\ref{cscont}) admits at least one solution in the sense 
of Definition \ref{weakdef}.  Moreover if $f_0$ is of the form (\ref{discini}) then $f$ is atomic of  form (\ref{remeq}), hence it  is a unique measure-valued solution to (\ref{cscont}).
\end{theo}

The part of the proof concerning the issue of existence follows from analysis of  approximation by atomic solutions
 originating from sums of Dirac's deltas, which correspond in the sense of Remark \ref{remcor} to solutions of (\ref{cs}). The main idea behind this approach is twofold.
First, we have  better  regularity of solutions of (\ref{cs}) for $\alpha<\frac{1}{2}$. 
It was connected to results from \cite{jps}, where we proved that for $0<\alpha<\frac{1}{2}$, system (\ref{cs}) admits a unique $W^{1,1}([0,T])$ solution $(x,v)$, which by Remark \ref{remcor} corresponds to 
a solution of (\ref{cscont}) in the sense of Definition \ref{weakdef}. Since in fact $\alpha\in(0,\alpha_0)$ for some $\alpha_0<\frac{1}{2}$, we can in fact  prove that $(x,v)$ is 
bounded in $W^{1,p}([0,T])$ for some $p>1$. 
 Such bound  provides equicontinuity of sequences of solutions of (\ref{cs}), which allows to extract a convergent subsequence. The second element of the proof 
is to change the way of looking at the alignment force term
\begin{align}\label{MR1}
\int_0^T\int_{\r^{2d}}F(f_n)f_n\nabla_v\phi dxdvdt,
\end{align}
where if $f_n\rightharpoonup f$ then it is not clear whether $F(f_n)f_n\rightharpoonup F(f)f$. It is useful to look at (\ref{MR1}) as
\begin{align*}
\int_0^T\int_{\r^{4d}}\psi(|x-y|)(w-v)\nabla_v\phi d\mu_ndt \mbox{ \ \ \  for $d\mu_n:= f_n(x,v,t)\otimes f_n(y,w,t)dxdvdydw$}.
\end{align*}
% for $\mu_n:= f_n(t,x,v)\otimes f_n(t,y,w)$.
% \begin{theo}\label{main}
% Let $0<\alpha<\frac{1}{2}$. For any compactly supported initial data $f_0\in {\mathcal M}$ and any $T>0$, Cucker-Smale's flocking model (\ref{cscont}) admits at least one solution in the sense of 
% Definition \ref{weakdef}.  Moreover if $f_0$ is of the form (\ref{discini}) then $f$ is of the form (\ref{remeq}) and is unique.
% \end{theo}

The uniqueness part of Theorem \ref{main3} is explained and proved in section \ref{sec5}. The main idea is to analyze possible support of 
measure-valued solutions for initial atomic configurations. It is related to point 5 of Definition \ref{weak} of weak solutions to (\ref{cscont}). We emphasize that we derive the propagation of the support described in point 5 from the behavior of the approximate solutions. The question whether these properties are exhibited by a wider class of possible weak solutions is out of the scope of this paper. Such generalization would require an  essentially different approach to
construction of solutions.

\begin{rem}\rm
Throughout the remaining part of the paper we assume without a loss of generality that the total mass of $f_0$ equals to $1$.
\end{rem}

Let us give  an overview of the proof of existence. Suppose that $f_0$ is a given, compactly supported measure belonging to ${\mathcal M}$ and assume without a loss of generality that
\begin{align}\label{wlog}
{\rm supp}f_0\subset B(R_0),
\end{align}
where $B(R_0)$ is a ball centered at $0$ with radius $R_0$. For such $f_0$ we take $f_{0,\epsilon}\in{\mathcal M}$ of the form
\begin{align}\label{eid}
f_{0,\epsilon}=\sum_{i=1}^Nm_i\delta_{x_{0,i}^\epsilon}\otimes\delta_{v_{0,i}^\epsilon},
\end{align}
which corresponds to the initial data $(x_{0,\epsilon},v_{0,\epsilon})$ to a particle system (\ref{cs}). Moreover we assume that
\begin{align*}
d(f_{0,\epsilon},f_0)\stackrel{\epsilon\to 0}{\longrightarrow} 0
\end{align*}
and that the support of $f_{0,\epsilon}$ is contained in $B(2R_0)$. The existence of such approximation is standard (we refer for example to the beginning of section 6.1 in \cite{haliu} for the details). 
Now suppose that $(x^n_\epsilon,v^n_\epsilon)$ is a solution to (\ref{cs}) with the communication weight
\begin{align}\label{psin}
\psi_n(s):=\min\{\psi(s), n\},
\end{align}
subjected to the initial data $(x_{0,\epsilon},v_{0,\epsilon})$, which by Remark \ref{remcor} means that
\begin{align}\label{disc}
f^n_\epsilon = \sum_{i=1}^Nm_i\delta_{x_{\epsilon,i}^n}\otimes\delta_{v_{\epsilon,i}^n}
\end{align}
is a solution of (\ref{cscont}) with the initial data $f_{0,\epsilon}$. Our goal  is to converge with $\epsilon$ to $0$ and 
with $n$ to $\infty$ to obtain a solution $f$ of equation (\ref{cscont}) subjected to the initial data $f_0$. 

The proof of existence can be summarized in the following steps:

\begin{description}
\item[Step 1.] Given $T>0$, for each $\epsilon$ and $n$, we prove existence of a solution $f^n_\epsilon$ corresponding to the initial data $f_{0,\epsilon}$ and satisfying various regularity properties.
\item[Step 2.] We take a sequence $f_n = f^n_\epsilon$ for $\epsilon=\frac{1}{n}$. Due to the conservation of mass and the regularity proved in step 1 we extract a subsequence ${f_{n_k}}$ converging
in $L^\infty(0,T;({\mathcal M}_+,d))$ to some $f\in L^\infty(0,T;{\mathcal M}_+)$.
\item[Step 3.] We converge with each term in the weak formulation for ${f_{n_k}}$ to the respective term in
 the weak formulation for $f$. This can be easily done for each term except the alignment force term i.e. the term
\begin{align}\label{3a}
\int_0^T\int_{\r^{2d}}F_{n_k}(f_{n_k})f_{n_k}\nabla_v\phi dxdvdt.
\end{align}
\item[Step 4.] Analysis of the alignment force term is not straightforward.  Formally we multiply an $L_p$ function by a measure, thus the result is not well defined at the very first sight. 
We replace (\ref{3a})  with an $n_k$-independently regular modification of the form
\begin{align*}
\int_0^T\int_{\r^{2d}}F_m(f_{n_k})f_{n_k}\nabla_v\phi dxdvdt.
\end{align*}
The error between the alignment force term and it's substitute will  be controlled in terms of $m$ and uniformly with respect to $n_k$.
\item[Step 5.] For such subsequence we converge with the substitute alignment force term to
\begin{align*}
\int_0^T\int_{\r^{2d}}F_m(f)f\nabla_v\phi dxdvdt.
\end{align*}
\item[Step 6.] We are then left with converging with the substitute alignment force term to the original alignment force term i.e. with $m\to\infty$.
We show that $F(f)$ is a measure with respect to the measure $df=fdxdv$.
\item[Step 7.] We finish the proof by making sure that each and every point of Definition \ref{weakdef} is satisfied by our candidate for the solution.
\end{description}
\begin{rem}\rm\label{refrem}
Before we proceed further, let us compare our strategy to the mean-field limit used for the CS model with regular weight like in \cite{haliu}. For Lipschitz continuous $\psi$, through standard argumentation, 
one can easily show well-possedness for the particle system \eqref{cs}. This includes Lipschitz continuous dependence of solutions with respect to the perturbations of the initial data. Such stability
 can in turn be translated into Lipschitz continuous dependence of measure-valued solutions with respect to perturbations of initial measures. Thus inequality
\begin{align}\label{stab}
d(f_1,f_2)\leq C(T)d(f_{0,1},f_{0,2})
\end{align}
can be obtained, where $f_1$ and $f_2$ are two solutions of the type \eqref{remeq} with initial data $f_{0,1}$ and $f_{0,2}$ of the form \eqref{discini}. Then  stability ensures  possibility
 to extract a convergent subsequence as presented after Remark \ref{wlog}. If one aims to apply such reasoning to the case with singular weight it turns out that the constant $C(T)$ from \eqref{stab} 
depends on the Lipschitz constant of $\psi$ on $[\delta,\infty)$, where $\delta$ is the lower bound of the distances between particles. However even for the initial data, the distance between
 particles $\delta$ converges to $0$ as $\epsilon\to 0$. Thus the Lipschitz constant of $\psi$ on $[\delta,\infty)$ converges to infinity and so does $C(T)$. There are ways to overcome this difficulty at least 
locally in time like in \cite{carchoha}, but for a global in time result stability of the type \eqref{stab} does not seem to work since  particles may eventually collide causing a blowup of $C(T)$. For the
 discussion on the possibility of collisions between particles in the singular case of $\alpha\in(0,1)$ we refer to \cite{jpe}, while in case of $\alpha\geq 1$ we refer to \cite{ccmp}.
\end{rem}

Let us state some various properties of the approximative solutions $f^n_\epsilon$. It is in fact the first step of the proof (as presented above) but since it is self-contained and quite lengthy 
we will present it in a form of separate proposition the proof of which can be found in Appendix \ref{app}. 

\begin{prop}\label{pop}
Given $T>0$.
Let $f_{0,\epsilon}$ be of the form (\ref{discini}). Then for each $n=1,2,...$, there exists a unique solution $f^n_\epsilon$ to kinetic equation (\ref{cscont}) that 
corresponds\footnote{See Remark \ref{remcor}.} to a smooth and classical solution $(x^n,v^n)$ of particle system (\ref{cs}). Moreover there exist $n$ and $\epsilon$ independent 
constants $M>0$ and $p>1$, such that the following conditions are satisfied:
\begin{enumerate}[(i)]
\item For all $t\in[0,T]$ and all $n$ and $\epsilon$ the total mass of $f^n_\epsilon$ i.e. the value $\int_{\r^{2d}}f^n_\epsilon dxdv$ is equal to $1$.
\item The support of $f^n_\epsilon$ is contained in a ball $B({\mathcal R})$, where ${\mathcal R}:=2R_0(T+1)$.
\item We have
\begin{align*}%\label{unif}
\int_0^T\sum_{i=1}^{N_n}m^n_{i,\epsilon}\left|\dot{v}^n_{i,\epsilon}\right|^pdt + 
\int_0^T\sum_{i,j=1}^{N_n}m_im_j\psi_n^p(|x^n_{i,\epsilon}-x^n_{j,\epsilon}|) |v^n_{i,\epsilon}-v^n_{j,\epsilon}|^pdt\leq M({\mathcal R}).
\end{align*}
\item We have
\begin{align*}%\label{unif2}
\int_0^T\sum_{i,j=1}^{N_n}m_im_j\psi_n^p(|x^n_{i,\epsilon}-x^n_{j,\epsilon}|) |v^n_{i,\epsilon}-v^n_{j,\epsilon}|dt\leq M({\mathcal R}).
\end{align*}
\item For each Lipschitz continuous and bounded $g:\r^{2d}\to\r$, we have
\begin{align*}%\label{unif3}
\left\|\frac{d}{dt}\int_{\r^{2d}}gf^n_\epsilon dxdv\right\|_{L^p([0,T])}\leq M_g(Lip(g),{\mathcal R})
\end{align*}
\end{enumerate}
\end{prop}
\begin{rem}\rm
Point $(iii)$ of Proposition \ref{pop} implies in particular that the sequence $(x^n_\epsilon,v^n_\epsilon)$ is uniformly bounded in $W^{1,p}([0,T])$. We mention this to keep the 
continuity with the idea of the proof presented at the beginning of this section.
\end{rem}

\begin{rem}\rm\label{xto0}
It is worthwhile to note that since by $(iii)$ from Proposition \ref{pop} the derivative of velocity $\dot{v}$ is uniformly integrable, then
\begin{align*}
|v^n_i(t)-v^n_i(0)|\leq\int_0^t|\dot{v}_i^n|ds\leq \omega(t)\to 0 \mbox{ \ \ as  \ } t\to 0.
\end{align*} 
Moreover the function $\omega$ is independent of $i$ and $n$.
\end{rem}

Proposition \ref{pop} is similar and was inspired by Theorem 3.1 from \cite{jps} that we present below for readers convenience.

\begin{theo}\label{jpsmain}
Let $\alpha\in(0,\frac{1}{2})$ be given. Then for all $T>0$ and arbitrary initial data there exists a unique solution $(x,v)$ to (\ref{cs}) with communication weight given by 
(\ref{psi}). Moreover $v\in W^{1,1}([0,T])$, and thus $x\in W^{2,1}([0,T])$.
\end{theo}

The proof of the above theorem can be found in \cite{jps}. Moreover its existence part follows also directly from Theorem \ref{main3} and Remark \ref{remcor}. Furthermore our argumentation
 from Section \ref{sec5} can be used to simplify the uniqueness part of the proof of Theorem \ref{jpsmain}. We present the simplified proof in Appendix \ref{app} for the sake of completeness.

\section{Proof of Theorem \ref{main3} (existence)}\label{sec4}
In this section we follow the steps presented in the previous section and prove the existence part of Theorem \ref{main3}.\\

{\bf Step 1.} From the very beginning we fix $T>0$.  Proposition \ref{pop} and Remark \ref{remcor} ensure the existence of $f^n_\epsilon$ with properties $(i)$-$(v)$ from Proposition \ref{pop}.
We solve particle system (\ref{cs}) with initial data (\ref{eid}) in the time interval $[0,T]$ under assumption that the communication weight is in form (\ref{psin}).
 By Proposition \ref{pop} we are ensured that for some $p>1$
\begin{align*}
 \|f^n_\epsilon\|_{L^\infty(0,T; \mathcal{M})} &= 1,\qquad 
\|F_n(f^n_\epsilon)f^n_\epsilon\|_{L^p(0,T;\mathcal{M})} \leq M(T).
\end{align*}

{\bf Step 2.} We take $\epsilon=\frac{1}{n}$ and denote $f_n:=f^n_\frac{1}{n}$. Since $f_n$ is of the form (\ref{disc}) it is clear that
\begin{align*}
\int_{\r^d\times \r^d} f_n dxdv = \sum_{i=1}^{N_n}m_{i,n} = 1.
\end{align*}
For each $n$ the function $f_n$ may be treated as a mapping from $[0,T]$ into the metric space $({\mathcal M}_+,d)$. For the purpose of showing that $f_n$ has a convergent subsequence 
we use Arzela-Ascoli theorem. We  make sure  $f_n$ is a bounded and equicontinuous sequence of functions with a relatively compact pointwise sequences $f_n(t)$. Uniform boundedness of $f_n$ is implied 
by the conservation of mass, 
while relative compactness of $f_n(t)$ follows from the uniform boundedness of $f_n(t)$ in TV topology and Corollary \ref{comp}. Finally in order to prove equicontinuity of $f_n$ we take 
arbitrary $s,t\in[0,T]$ and arbitrary Lipschitz continuous, bounded function $g$ with $Lip(g)\leq 1$ and $\|g\|_\infty\leq 1$ and use estimation $(v)$ from Proposition \ref{pop} to write
\begin{align}\label{bla}
\left|\int_{\r^{2d}}g(f_n(s)-f_n(t))dxdv\right|=\left|\int_t^s\frac{d}{dr}\int_{\r^{2d}} gf_ndxdvdr\right|=:\omega(|s-t|).
\end{align}
Point $(v)$ of Proposition \ref{pop} states that functions $t\mapsto\frac{d}{dt}\int_{\r^{2d}}gf_n(t)dxdv$ are uniformly bounded in $L^p([0,T])$ for some $p>1$, which in particular means that 
they are uniformly integrable. On the other hand it implies that the function $\omega$ is a {\it good} modulus of uniform continuity for the left-hand side of (\ref{bla}). Now since this estimation does 
not depend on the choice of $g$ (only on the choice of $Lip(g)$), it is also valid for the supremum over all $g$, which implies that
\begin{align*}
d(f_n(s),f_n(t))\leq\omega(|s-t|).
\end{align*}
The above inequality proves that the sequence of functions $t\mapsto f_n(t)$ is equicontinuous as a mapping from $[0,T]$ to $({\mathcal M},d)$ (recall the bounded--Lipschitz distance defined 
in (\ref{bld})). Thus the sequence $f_n$ satisfies the assumptions of Arzela-Ascoli theorem. Therefore there exists $f\in L^\infty(0,T;{\mathcal M}_+)$, such that up to a subsequence
\begin{align*}
\|d(f_n,f)\|_\infty\to 0.
\end{align*}
By $(ii)$ from Proposition \ref{pop} it implies that the support of $f$ is included in $B({\mathcal R})$.

{\bf Step 3.} After a brief look at the weak formulation for $f_n$ i.e. (\ref{weak}), we understand that since $f_n\to f$ in $L^\infty(0,T;({\mathcal M}_+,d))$, then in particular for $\phi\in {\mathcal G}$,
 we have
\begin{align*}
\int_0^T\int_{\r^{2d}}f_n[\partial_t\phi + v\nabla\phi]dxdvdt\to \int_0^T\int_{\r^{2d}}f[\partial_t\phi + v\nabla\phi]dxdvdt
\end{align*}
and
\begin{align*}
\int_{\r^{2d}}f_{0,\frac{1}{n}}\phi(\cdot,\cdot,0)dxdv\to\int_{\r^{2d}}f_{0}\phi(\cdot,\cdot,0)dxdv
\end{align*}
and the only problem is with the second term on the left-hand side of (\ref{weak}) i.e. the alignment force term
\begin{align}\label{ali2}
\int_0^T\int_{\r^{2d}}F_{n}(f_{n})f_{n}\nabla_v\phi dxdvdt.
\end{align}

{\bf Step 4.} To deal with the problem of convergence with the alignment force term we replace it in the following manner
\begin{align*}
\int_0^T\int_{\r^{2d}}f_n[\partial_t\phi + v\nabla\phi]dxdvdt + \int_0^T\int_{\r^{2d}}F_{m}(f_n)f_n\nabla_v\phi dxdvdt\\ = -\int_{\r^{2d}}f_{0,\frac{1}{n}}\phi(\cdot,\cdot,0)dxdv + {\mathcal J},
\end{align*}
where
\begin{align*}
{\mathcal J} := \int_0^T\int_{\r^{2d}}\left(F_{m}(f_n)-F_{n}(f_n)\right)f_n\nabla_v\phi dxdvdt
\end{align*}
for
\begin{align*}
F_{m}(f_n)(x,v,t):=\int_{\r^d\times\r^d}\psi_m(|x-y|)(w-v)f_n(y,w,t)dydw.
\end{align*}
However, as mentioned at the beginning of Section \ref{sec3}, instead of looking at (\ref{ali2}) as an integral of a product of $F_n(f_n)$ with $f_n$, we are going to see it as an integral of
\begin{align}\label{gn}
g_n(x,y,w,v):=\psi_n(|x-y|)(w-v)\nabla_v\phi(x,v,t)
\end{align}
with respect to the measure
\begin{align*}
d\mu_n(x,y,w,v,t):=f_n(x,v,t)\otimes f_n(y,w,t) dxdvdydw.
\end{align*}
By Fubini's theorem we have
\begin{align*}
&\int_0^T\int_{\r^{2d}}F_n(f_n)f_n\nabla_v\phi
 dxdvdt=\\
&=\int_0^T\int_{\r^{2d}}\left(\int_{\r^{2d}}\psi_n(|x-y|)(w-v)f(y,w,t)dydw\right)\nabla_v\phi(x,v,t)f(x,v,t)dxdvdt\\
&=\int_0^T\int_{\r^{4d}}g_nd\mu_ndt
\end{align*}
and a similar identity holds for $\int_0^T\int_{\r^{2d}}F_m(f_n)f_n\nabla_v\phi dxdvdt$. Therefore
\begin{align*}
{\mathcal J} = \int_0^T\int_{\r^{4d}}(g_m-g_n)d\mu_ndt.
\end{align*}
Moreover we have
\begin{align*}
g_m-g_n = 0
\end{align*}
in the set $\{(x,y,w,v):|x-y|>\max\{m^{-\frac{1}{\alpha}},n^{-\frac{1}{\alpha}}\}\}$, which provided that\footnote{Which we may assume since we are going to converge with $n\to\infty$ for each fixed $m$.} 
$n>m$ implies that
\begin{align}\label{foot1}
|g_m-g_n|\leq |g_n|\chi_{\{(x,y,w,v):|x-y|\leq m^{-\frac{1}{\alpha}}\}}.
\end{align}
Therefore for
\begin{align*}
A(m,n)&:=\left\{t: \int_{B(m,n)}|w-v|d\mu_n>m^{-\frac{1}{2}}\right\}, &
B(m,n)&:=\left\{(x,y,w,v):|x-y|\leq m^{-\frac{1}{\alpha}}\right\}
\end{align*}
we have
\begin{align*}
|{\mathcal J}|\leq C\left(\int_{A(m,n)}\int_{B(m,n)} |g_n|d\mu_n dt+ \int_{(A(m,n))^c}\int_{B(m,n)} |g_n|d\mu_ndt\right) =: I + II.
\end{align*}
Now if $|x-y|\leq m^{-\frac{1}{\alpha}}$ then $\psi_n(|x-y|)\geq \min\{m, n\}= m$ and for all $t\in A(m,n)$ we have
\begin{align*}
{\mathcal L_n}(t)&:=\int_{\r^{4d}}\psi_n(|x-y|)|w-v|d\mu_n \\
&\geq\int_{B(m,n)}\psi_n(|x-y|)|w-v|d\mu_n
\geq m\cdot\int_{B(m,n)}|w-v|d\mu_n> m^\frac{1}{2}. 
\end{align*}
Furthermore, integrating with respect to $d\mu_n$ reveals that
\begin{align*}
{\mathcal L_n}(t)=\sum_{i,j=1}^N\psi(|x_i^n(t)-x_j^n(t)|)|v_i^n(t)-v_j^n(t)|
\end{align*}
which by Proposition \ref{pop},$(iv)$ implies that the sequence ${\mathcal L_n}$ is uniformly bounded in $L^p([0,T])$ for some $p>1$ and thus -- it is uniformly integrable which further implies that
\begin{align}\label{ifin}
I\leq C\|\nabla_v\phi\|_\infty\int_{\{t:{\mathcal L_n}(t)>m^\frac{1}{2}\}}{\mathcal L}_n(t)dt\leq C(m)\|\nabla_v\phi\|_\infty\stackrel{m\to\infty}{\longrightarrow}0,
\end{align}
since $|{\mathcal L_n}(t)>m^\frac{1}{2}|\leq \frac{\|{\mathcal L_n}\|_{L^1}}{m^\frac{1}{2}}\to 0$ as $m\to\infty$.

To estimate $II$ we introduce the set $B_t(m,n)$ of those pairs $(i,j)$ such that $|x_{i}^n(t)-x_{j}^n(t)|\leq m^{-\frac{1}{\alpha}}$. Then by H\" older's inequality with exponent $q=\frac{1}{\theta}$, 
for some arbitrarily small $\theta>0$, we have
\begin{align}
II &\leq\|\nabla_v\phi\|_\infty\int_{(A(m,n))^c}\sum_{i,j\in B_t(m,n)}m_{i,n}m_{j,n} \psi_n(|x_{i}^n-x_{j}^n|)|v_{i}^n-v_{j}^n|dt\nonumber\\
&=\|\nabla_v\phi\|_\infty\int_{(A(m,n))^c}\sum_{i,j\in B_t(m,n)}(m_{i,n}m_{j,n})^{1-\theta} \psi_n(|x_{i}^n-x_{j}^n|)|v_{i}^n-v_{j}^n|^{1-\theta}\\
&\quad\quad\cdot (m_{i,n}m_{j,n})^{\theta}|v_{i}^n-v_{j}^n|^{\theta}dt\nonumber\\
&\leq \|\nabla_v\phi\|_\infty\left(\int_{(A(m,n))^c}\sum_{i,j\in B_t(m,n)}m_{i,n}m_{j,n} \psi_n^\frac{1}{1-\theta}(|x_{i}^n-x_{j}^n|)|v_{i}^n-v_{j}^n|dt\right)^{1-\theta} \nonumber\\ 
&\quad\quad\cdot \left(\int_{(A(m,n))^c}\sum_{i,j\in B_t(m,n)}m_{i,n}m_{j,n} |v_{i}^n-v_{j}^n|dt\right)^\theta\nonumber\\
&\leq \|\nabla_v\phi\|_\infty\left(\int_0^T\sum_{i,j=1}^{N_n}m_{i,n}m_{j,n} \psi_n^\frac{1}{1-\theta}(|x_{i}^n-x_{j}^n|)|v_{i}^n-v_{j}^n|dt\right)^{1-\theta}\\
&\quad\quad\cdot \left(\int_{(A(m,n))^c}\int_{B(m,n)}|w-v|d\mu_ndt\right)^\theta\nonumber\\
&\leq \|\nabla_v\phi\|_\infty\left(\int_0^T\sum_{i,j=1}^{N_n}m_{i,n}m_{j,n} \psi_n^\frac{1}{1-\theta}(|x_{i}^n-x_{j}^n|)|v_{i}^n-v_{j}^n|dt\right)^{1-\theta}\cdot \left(Tm^{-\frac{1}{2}}\right)^\theta.\label{iipom}
\end{align}
By Proposition \ref{pop}, $(iv)$ the first multiplicand on the right-hand side of (\ref{iipom}) is uniformly bounded, which implies that
\begin{align}\label{iifin}
II\leq C\|\nabla_v\phi\|_\infty\left(Tm^{-\frac{1}{2}}\right)^\theta \stackrel{m\to\infty}{\longrightarrow} 0.
\end{align}
Estimations (\ref{ifin}) and (\ref{iifin}) imply that
\begin{align*}
|{\mathcal J}|\leq C(m)\|\nabla_v\phi\|_\infty
\end{align*}
for some $n$-independent positive constant $C(m)$ such that $C(m)\to 0$ as $m\to\infty$.\\

{\bf Step 5.} 
Our next goal is to ensure that the convergence
\begin{align}\label{conval}
\int_0^T\int_{\r^{2d}}F_{m}(f_n)f_n\nabla_v\phi dxdvdt\to \int_0^T\int_{\r^{2d}}F_{m}(f)f\nabla_v\phi dxdvdt
\end{align}
holds for each $m$ and each $\phi\in{\mathcal G}$. Let us fix $\phi\in{\mathcal G}$ and $m=1,2,...$ . For $g_m$ defined in (\ref{gn}), we have
{\small
\begin{align}\label{foot2}
\left|\int_0^T\int_{\r^{2d}}F_{m}(f_n)f_n\nabla_v\phi dxdvdt-\int_0^T\int_{\r^{2d}}F_{m}(f)f\nabla_v\phi dxdvdt\right|=\left|\int_0^T\int_{\r^{4d}}g_m(d\mu_n-d\mu)dt\right|\nonumber\\
\leq
 \left|\int_0^T\int_{\r^{4d}}g_m[d(f_n\otimes f_n)-d(f_n\otimes f)]dt\right| + \left|\int_0^T\int_{\r^{4d}}g_m[d(f_n\otimes f)-d(f\otimes f)]dt\right| =: I+II.
\end{align}
}
Furthermore, again by Fubini's theorem
\begin{align*}
I = \left|\int_0^T\int_{\r^{2d}}\left(\int_{\r^{2d}}g_m(df_n-df)\right)df_ndt\right|
\end{align*}
and since for each $x,v$ the function $(y,w)\mapsto g_m(x,y,v,w)$ is Lipschitz continuous and bounded with $Lip(g_m)+\|g_m\|_\infty \leq C_1$ for some $C_1=C_1(m,\|\nabla_v\phi\|_\infty,Lip(\nabla_v\phi))$ 
then by Lemma \ref{bldist} we have
\begin{align*}
I\leq C_1\int_0^T\int_{\r^{2d}}d(f_n,f)df_n\leq C_1T\|d(f_n,f)\|_\infty\to 0\ \ \ {\rm as}\ \ \ n\to\infty.
\end{align*}
Similarly also $II\to 0$ with $n\to\infty$. This concludes the proof of convergence (\ref{conval}).\\

{\bf Step 6.}
At this point after converging with $n$ to infinity we are left with the weak formulation for $f$ that reads as follows:
\begin{align*}
\int_0^T\int_{\r^{2d}}f[\partial_t\phi + v\nabla\phi]dxdvdt &+ \int_0^T\int_{\r^{2d}}F_{m}(f)f\nabla_v\phi dxdvdt\\ &= -\int_{\r^{2d}}f_{0}\phi(\cdot,\cdot,0)dxdv + {\mathcal J}(m)
\end{align*}
for all $m=1,2,...$ and all $\phi\in {\mathcal G}$ with
\begin{align*}
{\mathcal J}(m)\to 0\ \ \ {\rm as}\ \ \ m\to\infty.
\end{align*}
Therefore it suffices to show that
\begin{align}\label{poem}
\int_0^T\int_{\r^{2d}}F_{m}(f)f\nabla_v\phi dxdvdt\to \int_0^T\int_{\r^{2d}}F(f)f\nabla_v\phi dxdvdt.
\end{align}
By Fubini's theorem for
\begin{align*}
d\mu &=(f\otimes f)(x,v,y,w,t)dxdvdydw,\\
g_m &=\psi_{m}(|x-y|)(w-v)\nabla_v\phi,\qquad
g =\psi(|x-y|)(w-v)\nabla_v\phi,
\end{align*}
we have
\begin{align}\label{fubub}
\int_0^T\int_{\r^{2d}}F_{m}(f)f\nabla_v\phi dxdvdt &= \int_0^T\int_{\r^{2d}}g_m d\mu dt,\nonumber\\
\int_0^T\int_{\r^{2d}}F(f)f\nabla_v\phi dxdvdt &= \int_0^T\int_{\r^{2d}} gd\mu dt
\end{align}
provided that the integral on the right-hand side of (\ref{fubub}) is well defined. Therefore to show (\ref{poem}) it suffices to prove that
% \begin{align*}
$g_m\to g$
% \end{align*}
in $L^1$ with respect to measure $d\mu$. To prove this we first show that
% \begin{align*}
$g_m\to g$
% \end{align*}
a.e. with respect to the measure $d\mu$. Clearly the convergence holds on
\begin{align*}
A:=\{(x,v,y,w,t): x\neq y\}\cup \{(x,v,y,w,t): x=y, v=w\}
\end{align*}
and it suffices to show that the set $A^c=\{(x,v,y,w,t): x=y, v\neq w\}$ is of measure $d\mu$ zero. We have $\psi_m\equiv m$ on $A^c$ and thus
\begin{multline*}
I_m:= \int_0^T\int_{\r^{4d}}|g_m|d\mu dt =  \int_0^T\int_{\r^{4d}}\psi_{m}(|x-y|)|w-v||\nabla_v \phi|d\mu dt \\
\geq \int_{A^c} \psi_{m}(|x-y|)|w-v||\nabla_v\phi| d\mu dt= \int_{A^c} m|w-v||\nabla_v\phi|d\mu dt= m\int_{A^c} |w-v||\nabla_v\phi|d\mu dt.
\end{multline*} 
Thus either
\begin{align}\label{alter}
I_m\to\infty\ \ \  {\rm or}\ \ \ \int_{A^c}|w-v||\nabla_v\phi|d\mu = 0.
\end{align}
The proofs of Step 4 and Step 5 remain true if we substitute $g_m$ and $g_n$ with $|g_m|$ and $|g_n|$ respectively\footnote{Indeed, since $||g_m|-|g_n||\leq|g_m-g_n|$ we may replace in 
(\ref{foot1}) $g_m$ and $g_n$ with $|g_m|$ and $|g_n|$ and proceed with the proof in the same way as in Step 4. On the other hand in Step 5, the convergence of $I$ and $II$ from (\ref{foot2}) was a
 result of that $\|d(f_n,f)\|_\infty\to 0$ and that $g_m$ is a Lipschitz continuous function, which remains true for $|g_m|$.}. Therefore also the respective convergences hold for $|g_m|$ and $|g_n|$, yielding
\begin{align}\label{j1}
\left|\int_0^T\int_{\r^{4d}}|g_m|d\mu_ndt - \int_0^T\int_{\r^{4d}}|g_n|d\mu_n dt\right|\leq C(m)\|\nabla_v\phi\|_\infty\stackrel{m\to\infty}{\longrightarrow} 0
\end{align}
and
\begin{align}\label{j2}
\left|\int_0^T\int_{\r^{4d}}|g_m|d\mu_ndt - \int_0^T\int_{\r^{4d}}|g_m|d\mu dt\right|\stackrel{n\to\infty}{\longrightarrow}0.
\end{align} 
Moreover for each $m$ and $n$, we have
\begin{align*}
I_m &\leq \left|\int_0^T\int_{\r^{4d}}|g_m|d\mu dt - \int_0^T\int_{\r^{4d}}|g_m|d\mu_ndt\right|\\
 &+ \left|\int_0^T\int_{\r^{4d}}|g_m|d\mu_ndt-\int_0^T\int_{\r^{4d}}|g_n|d\mu_ndt\right|
 + \int_0^T\int_{\r^{4d}}|g_n|d\mu_ndt.
\end{align*}
Now, (\ref{j2}) implies that for each $m$ we may choose $n$ big enough, so that
\begin{align*}
\left|\int_0^T\int_{\r^{4d}}|g_m|d\mu dt - \int_0^T\int_{\r^{4d}}|g_m|d\mu_ndt\right|\leq 1.
\end{align*}
Furthermore, by (\ref{j1}) for such $n$ we have
\begin{align*}
\left|\int_0^T\int_{\r^{4d}}|g_m|d\mu_ndt-\int_0^T\int_{\r^{4d}}|g_n|d\mu_ndt\right|\leq |{\mathcal J}(m)|
\end{align*}
and finally by estimation $(iii)$ from Proposition \ref{pop}
\begin{align*}
\int_0^T\int_{\r^{4d}}|g_n|d\mu_ndt &\leq 
\|\nabla_v\phi\|_\infty \int_0^T\int_{\r^{4d}}\psi_n(|x-y|)|w-v|d\mu_ndt\\
&= \|\nabla_v\phi\|_\infty \int_0^T\sum_{i,j=1}^{N_n}m^n_im^n_j\psi_n(|x_i^n-x_j^n|)|v_j^n-v_i^n|dt\leq M
\end{align*}
and thus
\begin{align}\label{ex}
I_m\leq 1+|{\mathcal J}(m)|+M\leq C_2
\end{align}
for some positive constant $C_2$.
Therefore (\ref{alter}) and (\ref{ex}) imply that $\int_{A^c}|w-v||\nabla_v\phi| d\mu = 0$ and since the function $|w-v|$ is positive on $A^c$, then by a standard density argument $A^c$ is
 of measure $\mu$ zero and we have proved that
\begin{align*}
\psi_{m}(|x-y|)(w-v)\nabla_v\phi\to\psi(|x-y|)(w-v)\nabla_v\phi \mbox{ \ $\mu$-a.e.},\\
\psi_{m}(|x-y|)|w-v||\nabla_v\phi|\to\psi(|x-y|)|w-v||\nabla_v\phi| \mbox{  \ $\mu$-a.e.}.
\end{align*}
 Moreover by Fatou's lemma
\begin{align}\label{p4}
\int_0^T\int_{\r^{2d}}\psi(|x-y|)|w-v||\nabla_v\phi|d\mu dt &\leq \liminf_{m\to\infty} \int_0^T\int_{\r^{2d}}\psi_{m}(|x-y|)|w-v||\nabla_v\phi|d\mu dt\nonumber\\
&= \liminf_{m\to\infty} I_m \leq C_2.
\end{align}
Therefore the function $(x,y,v,w,t)\mapsto \psi(|x-y|)|w-v||\nabla_v\phi|$ belongs to $L^1(d\mu)$. This function is a proper dominating function for $\psi_{m}(|x-y|)(w-v)\nabla_v\phi$ and
 by the dominated convergence theorem we have (\ref{poem}) and the proof of step 6 is finished. 

{\bf Step 7.}
Let us now wrap up the proof and compare Definition \ref{weakdef} with what we were able to prove about $f$. We took an arbitrary initial data $f_0\in {\mathcal M}_+$ and proved existence of 
$f\in L^\infty(0,T;{\mathcal M}_+)$. Moreover in step 2 using estimates $(ii)$ and $(v)$ from Proposition \ref{pop} we proved that actually ${\rm supp}f\subset B({\mathcal R})$ and (point 1 
of Definition \ref{weakdef}). Point 2 of Definition \ref{weakdef} is an immediate consequence of $(ii)$ from Proposition \ref{pop}, while point 3 was the main focus of all the steps of the 
proof and it was finally proved in step 6. Point 4 of Definition \ref{weakdef} follows from (\ref{p4}) and Fubini's theorem. As a consequence of the weak formulation for $f$ we conclude that 
also $\partial_tf\in L^{p}(0,T;(C^1(B({\mathcal R})))^*)$. We are left with point 5 of Definition \ref{weakdef}. Suppose that $B(R)$ and $B(r)$ are two concentric balls, such that (\ref{ccentr}) 
is satisfied. Then the construction of $f_
 {0,n}$ ensures that
\begin{align*}
{\rm supp}f_{0,n} \cap {B\left(R-\frac{1}{n}\right)}\subset B\left(r+\frac{1}{n}\right)
\end{align*}
and for sufficiently large $n$ we have $r+\frac{1}{n}<r+\frac{R-r}{8}<R-\frac{R-r}{8}$. Translating it according to $(\ref{disc})$ we write that in the set ${\mathcal I}$ of those $i$ that 
$(x^n_{0,i},v^n_{0,i})\in B(R-\frac{R-r}{8})$ we actually have $(x^n_{0,i},v^n_{0,i})\in B(r+\frac{R-r}{8})$. By $(ii)$ and $(iii)$ from Proposition \ref{pop} (and in particular by Remark 
\ref{xto0}), for each $i\in {\mathcal I}$ and for each sufficiently big $n$, we have the $n$ independent bounds:
\begin{align*}
|x^n_i(t)|\leq |x^n_{0,i}| + t{\mathcal R}\stackrel{t\to 0}{\longrightarrow} |x^n_{0,i}|,\qquad
|v^n_i(t)|\leq |v^n_{0,i}| + \omega(t)\stackrel{t\to 0}{\longrightarrow} |v^n_{0,i}|.
\end{align*}
The above bounds, for sufficiently small $t$ imply that $(x^n_{i}(t),v^n_{i}(t))\in B(r+\frac{R-r}{6})$ as long as $i\in{\mathcal I}$. Similarly for $i\notin{\mathcal I}$ in a sufficiently small
 neighborhood of $t=0$, we have $(x^n_{i}(t),v^n_{i}(t))\notin B(R-\frac{R-r}{6})$. Therefore
\begin{align*}
{\rm supp}f_n(t)\cap B\left(R-\frac{R-r}{6}\right)\subset B\left(r+\frac{R-r}{6}\right)
\end{align*}
for sufficiently large $n$ and sufficiently small $t$. Thus we may pass to the limit with $n\to\infty$ to obtain (\ref{propag}). 
 This finishes the proof of the existence part of Theorem \ref{main3}.

\section{Proof of Theorem \ref{main3} (weak-atomic uniqueness)}\label{sec5}
In what follows we aim at proving that if initial configuration $f_0$ is an atomic measure, i.e. it satisfies (\ref{discini}), then  solution $f$ in the sense of Definition \ref{weakdef} is of the form (\ref{remeq}),
 and it is unique. We will base the proof on a very careful analysis of the local propagation of the support of $f$ that comes from point 5 of Definition \ref{weakdef}. What, we basically need, is that any
 amount of the mass $f$ that is separated from the rest of the mass remains separated at least for some time. 
It is required  to refine this property by adding a control over the shape in which the support in the $x$ and $v$ coordinates propagates. 
The difficulty comes from the fact that in the case of the particle system  the position 
$x_i$ of $i$th particle changes with its own unique velocity $v_i$. However  in the case of the kinetic equation
  characteristics are not  well defined. 

%%%%%%%%%%%%%%%%%%%%%%%%%

{\bf Step 1.}
By  point 1 in Definition \ref{weakdef} it is sufficient to prove the proposition only in an arbitrarily small neighborhood of $t=0$. Let $f_0$ be of the form (\ref{discini}). 
Our first task is to restrict $f_0$ to small balls with one particle (say $i$th particle) in $\r^{2d}$. Then we will use the local propagation of the support to prove that the mass that initially formed the
 $i$th particle remains atomic in some right-sided neighborhood of $t=0$. Since
\begin{align}\label{id5}
f_0=\sum_{i=1}^N m_i\delta_{x_{0,i}}\otimes\delta_{v_{0,i}}
\end{align}
for  number of atoms $N$, we have a finite number of initial positions and velocities of the particles $(x_{0,i},v_{0,i})$ for $i=1,...,N$, which implies that there exists $R_1>0$ such that for all 
$r_0<R_1$, we have
\begin{align}\label{oddziel}
f_0|_{B_i(r_0)} = m_i\delta_{x_{0,i}}\otimes\delta_{v_{0,i}}
\end{align}
for $B_i(R):=B_{x,v}((x_{0,i},v_{0,i}),r_0)$. 

At this point let us concentrate on one  atom,  we fix ${i}$. We aim at  showing that there exists $T^*$ such that 
\begin{align}\label{fd}
f^D:=f|_{B_i(\frac{r_0}{4}) } = m_i\delta_{x_i(t)}\otimes\delta_{v_i(t)}
\end{align}
in $[0,T^*]$ for some $\r^d$ valued functions $x_i$ and $v_i$. We emphasize that $r_0$ and $T^*(r_0)$ can be chosen to be arbitrarily small.
 Identity (\ref{oddziel}) implies that for any $0<r<r_0$, we have
\begin{align*}
{\rm supp}f_0\cap B_i(r_0)\subset B_i(r)
\end{align*}
which by point 5 of Definition \ref{weakdef} ensures that there exists $T^*$ such that
\begin{align}\label{sepa}
{\rm dist}\{{\rm supp}f^D(t),{\rm supp}f^C(t)\}>\frac{r_0}{8}
\end{align}
for all $t\in[0,T^*]$,
where $f^C(t):=f(t)-f^D(t)$.
Then one can find a smooth function $\eta:\r^{2d} \times [0,T_*] \to [0,1]$ such that $\eta\equiv 1$ over the support of $f^D$ and $\eta\equiv 0$ over the support of $f^C$. We have then $f^D\eta=f^D$. 
All these properties allow to state the following equation satisfied by $f^D$ on $[0,T^*]$:
\begin{align}\label{weakfd}
\partial_t f^D + v\cdot\nabla_xf^D + \dv[(F(f^C)+F(f^D))f^D] = 0.
\end{align}
This equation is satisfied in the same sense that (\ref{weak}) from Definition \ref{weakdef}. To prove that $f^D$ is indeed of form (\ref{fd}) we introduce
\begin{align}\label{xva}
\left\{
\begin{array}{rcl}
\frac{d}{dt}x_a(t)&=&v_a(t)\\
\frac{d}{dt}v_a(t)&=&\displaystyle \int_{\r^{2d}}\psi(|x_a(t)-y|)(w-v_a(t))f^C\,dydw
\end{array}
\right.
\end{align}
with the initial data $(x_a(0),v_a(0))=(x_{0,i},v_{0,i})$. Condition (\ref{sepa}) ensures that the right-hand side of $(\ref{xva})_2$ is 
smooth and thus (\ref{xva}) has exactly one smooth solution in $[0,T^*]$. Our goal is to show that $f^D$ is supported on the curve $(x_a(t),v_a(t))$ 
and that in fact (\ref{fd}) holds with $(x_i(t),v_i(t))\equiv(x_a,v_a)$. Since this feature will hold for all atoms, the whole $f$ will then be atomic.

%%%%%%%%%%%%

\smallskip 

%%%%%%%%%%%%%

{\bf Step 2.} In the next step we characterize possible evolution of  the support of the weak solution to (\ref{weakfd}).

\begin{lem}\label{excon}
Let $f$ be a weak solution to (\ref{cscont}) in the sense of Definition \ref{weakdef}. Assume further that $f$ has the structure of $f=f^D+f^C$ and fulfills the weak formulation of (\ref{weakfd}), and
\begin{align*}
{\rm supp}f_0^D=(x_0,v_0)
\end{align*}
for some given $(x_0,v_0)$. Then for any $R>0$ there exists $T^*$, such that
\begin{align*}
{\rm supp}f^D(t)\subset (x_0,v_0) + (tB_x(v_0,\epsilon))\times B_v(0,R)
\end{align*}
for all $t\in[0,T^*]$, with $\epsilon:=\sqrt{2R(R+|v_0|)}$, which can be arbitrarily small depending on smallness of $R$.
\end{lem}

To prove Lemma \ref{excon} it is required to show the following result.

\begin{lem}\label{conelem}
Let $f^D$ be a weak solution to (\ref{weakfd}) in the sense of Definition \ref{weakdef}. Assume further that there exists $T^*$, such that
\begin{align}\label{cone-1}
{\rm supp}f^D(t)\subset B((x_0,v_0),R)
\end{align}
for some given $(x_0,v_0)$ and $R>0$ and all $t\in[0,T^*]$. Then
\begin{align}\label{cone}
{\rm supp}f^D(t)\subset {\rm supp}f_0^D + tB_x(v_0,R)\times B_v(0,R). \end{align} It means that the support in the $x$-coordinates propagates in a cone defined by the ball $B_x(v_0,R)$ in direction $v_0$.
\end{lem}

% \begin{rem}\rm
% Lemma \ref{conelem} is quite similar to Lemma \ref{excon}. The difference is that in Lemma \ref{conelem} we prove that the support of $f$ propagates inside cone-shaped neighbourhood of the support of $f_0$, while in Lemma \ref{excon} we prove a little bit more, namely, that the support not only propagates inside such cone-shaped neighbourhood but also actually travels in the direction of the cone's axis.
% \end{rem}

\begin{proof}[Proof of Lemma \ref{conelem}]
Without a loss of generality we assume that $(x_0,v_0)=(0,0)$. The boundedness of the support in the $v$-coordinates is trivial and thus we focus  on the support in the $x$-coordinates.
 Suppose that $x_1\in \r^d$ and $\rho>0$ are such that
\begin{align*}
{\rm supp}f_0^D\cap B_x(x_1,\rho)\times\r^d = \emptyset
% \end{align*}
\mbox{ \ \ and let \ \ }
% \begin{align*}
\phi(x,t):=((\rho-Rt)^2-|x-x_1|^2)_+.
\end{align*}
Hence
\begin{align}
 {\rm supp}\, \phi(\cdot,t) = \{ |x-x_1| \leq |\rho - Rt|\}.
\end{align}
We test (\ref{weakfd}) by $\phi^2$ and integrate over the time interval $[0,T^*]$, obtaining
\begin{align*}
\int_{\r^{2d}}f^D(T^*)\phi(T^*)^2dxdv + 4\int_0^{T^*}\int_{\r^{2d}} f^D\phi[(\rho-Rt)R-(x-x_1)v]dxdvdt =\\ = \int_{\r^{2d}}f_0^D\phi(0)^2dxdv = 0.
\end{align*}
Since the first term on the left-hand side of the above equality is nonnegative, we have
\begin{align*}
\int_0^{T^*}\int_{\r^{2d}} f^D\phi[(\rho-Rt)R-(x-x_1)v]dxdvdt\leq 0.
\end{align*}
But for the interior of  the support of $\phi$, we have $\rho-Rt > |x-x_1|$ and by (\ref{cone-1}) $R
 >|v|$. It implies that
\begin{align*}
0< (\rho-Rt)R-(x-x_1)v, \mbox{ \  \ and\ hence \ } f\phi\equiv 0.
\end{align*}
This way we proved that in the complement of the support in $x$ of $f(t)$ lay all the balls centered outside of ${\rm supp}f_0$ and with a radius equals to $\rho-Rt$, which implies (\ref{cone}).
\end{proof}

%%%%%%%%%%%%%%%%

\begin{proof}[Proof of Lemma \ref{excon}]
We base the proof on Lemma \ref{conelem}. First we  establish proper $R$ and $T^*$. Since $f_0^D$ is concentrated in one point $(x_0,v_0)$ then for arbitrarily small $\rho$ 
\begin{align*}
{\rm supp}\,f_0^D\subset B((x_0,v_0),\rho).
\end{align*}
Now, Definition \ref{weakdef} point 5 ensures that there exist $R(\rho)$ and $T^*(\rho)$ such that
\begin{align*}
{\rm supp}f^D(t)\subset B((x_0,v_0),R)
\end{align*}
in $[0,T^*]$ and $R$ can be chosen  arbitrarily small (then also $T^*$ is small but still positive). We fix such $R$ and $T^*$ and note that we may apply Lemma \ref{conelem} on $[0,T^*]$. 
Without a loss of generality we assume that $x_0=0$ and test (\ref{weakfd}) with the function $\phi^2$, where
\begin{align*}
\phi(x,t):=((x-v_0t)^2-(t\epsilon)^2)_+
\end{align*}
and
\begin{equation}
 {\rm supp}\, \phi(\cdot, t)=\{ x\in \r^d: |x-v_0 t|\geq t\epsilon \}.
\end{equation}
By (\ref{ugh}), we have
\begin{align}\label{u}
0 &=\int_{\r^{2d}}f^D(t)\phi^2(t)dxdv - 4\int_0^t\int_{\r^{2d}}f^D \phi [-v_0(x-v_0t)-t\epsilon^2 + v(x-v_0t)]dxdvdt\nonumber\\
&\geq 4\int_0^t\int_{\r^{2d}}f^D \phi [t\epsilon^2 - (v-v_0)(x-v_0t)]dxdvdt.
\end{align}
On the support of $f^D$, we have $|v-v_0|\leq R$ and by Lemma \ref{conelem} it holds 
\begin{align*}
|x-v_0t|\leq |x-\underbrace{x_0}_{=0}|+ |v_0|t\leq t(|v_0|+R) + t|v_0| \leq t (2|v_0|+R).
\end{align*}
Hence, in view of definition of $\epsilon$, we conclude
\begin{align*}
(v-v_0)(x-v_0t)\leq (2|v_0|+R)R\, t < t \epsilon^2.
\end{align*}
Therefore the integrand on the right-hand side of (\ref{u}) is nonnegative, which means that it has to be equal to $0$, which further implies that
\begin{align*}
f^D\phi \equiv 0 \mbox{ \ \ in \  }[0,T^*].
\end{align*}
By the definition of $\phi$ it follows that $f^D(t)$ vanishes outside of the cone balls $tB_x(v_0,\epsilon)\times\r^d$. The lemma is proved.

\end{proof}

\smallskip

{\bf Step 3.} In this part we show that $f$ initiated by a state of (\ref{id5}) stays indeed atomic for all time.

\begin{prop}
Let $f$ be a solution to (\ref{weakfd}) in the sense of Definition \ref{weakdef}. Then if $f_0$ is of  form (\ref{discini}) then $f$ is an atomic solution (of  form (\ref{remeq}))
 and it is unique.
\end{prop}

\begin{proof} We show separately for each of atoms that each initial particle generates a mono-atomic solution (at least locally in time). Finiteness of  number of atoms allows to conclude
that the whole solutions is atomic. Hence, we study (\ref{weakfd}) with a mono-atomic initial data located in $(x_0,v_0)$.

We test (\ref{weakfd}) by $(v-v_a(t))^2$ getting
\begin{align}\label{duzoi}
&\frac{d}{dt}\int_{\r^{2d}}f^D(v-v_a(t))^2dxdv = -2\int_{\r^{2d}}f^D(v-v_a(t)) \dot{v}_a(t)dxdv\nonumber\\
&\quad\quad+2\int_{\r^{2d}}F(f^C)f^D(v-v_a(t)) dxdv +2 \int_{\r^{2d}}F(f^D)f^D(v-v_a(t)) dxdv\nonumber \\
&\quad\quad= -2I + 2II + 2III.
\end{align}
First we  deal with $III$. By symmetry of $f^D\otimes f^D$ with respect to $(x,v)$ and $(y,w)$, we have
\begin{align*}
III&=\int_{\r^{4d}}\psi(|x-y|)(w-v)f^Df^D(v-v_a(t)) dxdvdydw \\
&=\int_{\r^{4d}}\psi(|x-y|)(v-w)f^Df^D(w-v_a(t)) dxdvdydw\\ &=
\frac{1}{2}\int_{\r^{4d}}\psi(|x-y|)(w-v)f^Df^D(v-w) dxdvdydw\\
&=-\frac{1}{2}\int_{\r^{4d}}\psi(|x-y|)(w-v)^2f^Df^Ddxdvdydw\leq 0.
\end{align*}
Next let us take a closer look at $II$. By the definition of $F(f^C)$ 
\begin{align*}
II&=\int_{\r^{4d}}\psi(|x-y|)(w-v)f^Df^C(v-v_a(t))dxdvdydw\\
&=\int_{\r^{4d}}\psi(|x-y|)(w-v_a(t)+v_a(t)-v)f^Df^C(v-v_a(t))dxdvdydw\\
&=\int_{\r^{4d}}\psi(|x-y|)(w-v_a(t))f^Df^C(v-v_a(t))dxdvdydw\\
&\underbrace{-\int_{\r^{4d}}\psi(|x-y|)f^Df^C(v-v_a(t))^2dxdvdydw}_{\leq 0}\\
&\leq \int_{\r^{4d}}\psi(|x-y|)(w-v_a(t))f^Df^C(v-v_a(t))dxdvdydw =: II_2.
\end{align*}
Now we compare $II_2$ with $I$:
\begin{align}\label{ii2}
\left|II_2-I\right| &= \left|\int_{\r^{4d}}(\psi(|x_a(t)-y|)-\psi(|x-y|))(w-v_a(t))f^Df^C(v-v_a(t))dxdvdydw\right|\nonumber\\
&\leq\int_{\r^{4d}}\big|\psi(|x_a(t)-y|)-\psi(|x-y|)\big||w-v_a(t)|f^Df^C|v-v_a(t)|dxdvdydw.
\end{align}
The main problem with estimating the right-hand side of the above inequality lays in the estimation of
\begin{align*}
\big|\psi(|x_a(t)-y|)-\psi(|x-y|)\big|.
\end{align*}
This is the place where the separation of  supports explained by Lemma \ref{excon} comes into play. Both $(x_a(t),v_a(t))$ and $(x,v)$ are in the support of $f^D$, while $(y,w)$ is in the support 
of $f^C$. Thus (\ref{sepa}) implies that either

\begin{align}\label{ok}
|x-y|>\frac{r_0}{8}\ \ \ {\rm and}\ \ \ |x_a(t)-y|>\frac{r_0}{8}
\end{align}
or
\begin{align}\label{nieok}
|v-w|>\frac{r_0}{8}\ \ \ {\rm and}\ \ \ |v_a(t)-w|>\frac{r_0}{8}.
\end{align}
We handle the above two cases separately. 

In case (\ref{ok}) it is clear that
\begin{align}\label{okes}
|\psi(|x_a(t)-y|)-\psi(|x-y|)|\leq L|x-x_a(t)| = Lt^\frac{1}{2}\frac{|x-x_a(t)|}{t^\frac{1}{2}}.
\end{align}
for some constant $L=L(r_0)>0$, since $\psi$ is smooth outside of any neighborhood of $0$. 

In case of (\ref{nieok}) we are actually in a situation when at $t=0$ multiple particles are situated in the same spot with different velocities i.e. $f^C$ is divided into two parts
 $f^{C_1}$ and $f^{C_2}$. The first part submits to the same bounds as (\ref{ok}) while for the second, $f^{C_2}$, we have
\begin{align*}
f^{C_2}(0)=\sum_{j}m_j\delta_{x_{0,i}}\otimes\delta_{v_{0,j}} =:\sum_jf^{C_2}_j(0).
\end{align*}
Thus, initially $f^{C_2}$ is concentrated in the same position as $f^D$ but with different velocities. In this case we  apply Lemma \ref{excon} multiple times 
(once for $f^D$ and multiple times for each $f^{C_2}_j$). Even though Lemma \ref{excon} is written for solutions of (\ref{weak}) we may still apply it for $f^D$ and 
each of $f^{C_2}_j$, since the proof does not involve directly the dependence on $v$. 
Therefore, by Lemma \ref{excon}, we have
\begin{align*}
{\rm supp}f^D(t)\subset (x_{0,i},v_{0,i}) + tB_x(v_{0,i},\epsilon)
% \end{align*}
\mbox{ \ \ and \ \ } 
% \begin{align*}
{\rm  supp}f^{C_2}_j(t)\subset (x_{0,i},v_{0,j}) + tB_x(v_{0,j},\epsilon).
\end{align*}
At this point we fix $R>0$ and $T^*$ from Lemma \ref{excon}, so that $\epsilon$ is small enough that
\begin{align*}
B_x(v_{0,i},\epsilon)\cap  B_x(v_{0,j},\epsilon) = \emptyset \mbox{ \ for \ } i \neq j.
\end{align*}
Moreover
\begin{align*}
{\rm dist}(B_x(v_{0,i},\epsilon), B_x(v_{0,j},\epsilon))>C(r_0)>0.
\end{align*}
Again, we used that the number of all atoms is finite.
If so, then also
\begin{align*}
|x-y|>tC(R)\ \ \ {\rm and}\ \ \ |x_a(t)-y|>tC(R)
\end{align*}
for $x\in{\rm supp}f^D$ and $y\in{\rm supp}f^{C_2}$. Therefore in such case ($\psi(|s|)=s^{-\alpha}$ and $\psi'(|s|) \sim s^{-1-\alpha}$)
\begin{align}\label{nieokes}
\big|\psi(|x_a(t)-y|)-\psi(|x-y|)\big| \leq C(R)t^{-1-\alpha}|x-x_a(t)| = C(R)t^{-\frac{1}{2}-\alpha}\frac{|x-x_a(t)|}{t^\frac{1}{2}}.
\end{align}
We combine inequalities (\ref{ii2}), (\ref{okes})\footnote{Here is the entire estimation in case (\ref{ok}) and the estimation of $f^{C_1}$ in case (\ref{nieok}).} and (\ref{nieokes}) with 
the global bounds on the support of $f$ obtaining
\begin{align*}
\left|II_2-I\right|\leq A(t)\int_{\r^{2d}}t^{-\frac{1}{2}}|x-x_a(t)|\,|v-v_a(t)|f^Ddxdv
\end{align*}
for $A:=Lt^\frac{1}{2}+C(R)t^{-\frac{1}{2}-\alpha}$, which thanks to the fact that $\alpha<\frac{1}{2}$ is integrable with respect to $t$ over $[0,T^*]$. 
Taking into the account our estimations of $I$, $II$ and $III$ we come back to (\ref{duzoi}) and claim that
\begin{align}\label{fzv}
\frac{d}{dt}\int_{\r^{2d}}f^D|v-v_a(t)|^2dxdv&\leq A(t)\int_{\r^{2d}}t^{-\frac{1}{2}}|x-x_a(t)|\,|v-v_a(t)|f^Ddxdv\\
&\leq A(t)\left(\int_{\r^{2d}}f^Dt^{-1}|x-x_a(t)|^2dxdv + f^D|v-v_a(t)|^2\right).\nonumber
\end{align}
To finish the proof there is a  need to estimate the first integrand on the right-hand side of (\ref{fzv}). We test\footnote{Even though $|x-x_a(t)|^2t^{-1}$ is not a good test
 function for (\ref{weakfd}), we can approximate the singularity at $t=0$ by modification $(t+l)^{-1}$ and then let $l\to 0$.} (\ref{weakfd}) with $|x-x_a(t)|^2t^{-1}$ getting
\begin{align*}
\frac{d}{dt}\int_{\r^{2d}}t^{-1}f^D|x-x_a(t)|^2dxdv &+ \int_{\r^{2d}}t^{-2}f^D|x-x_a(t)|^2dxdv\\
&\leq 2\int_{\r^{2d}}t^{-1}f^D(x-x_a(t))\dot{x}_a(t)dxdv\\
&\quad\quad - 2\int_{\r^{2d}}t^{-1}f^D(x-x_a(t))vdxdv.
\end{align*}
and apply Young's inequality with $\delta >0$ to obtain
\begin{align}\label{fzx}
\frac{d}{dt}\int_{\r^{2d}}t^{-1}f^D|x-x_a(t)|^2dxdv &+ \int_{\r^{2d}}t^{-2}f^D|x-x_a(t)|^2dxdv\nonumber\\
&\leq 2\int_{\r^{2d}}t^{-1}f^D|x-x_a(t)||v-v_a(t)|dxdv\nonumber\\
&\leq \delta \int_{\r^{2d}}t^{-2}f^D|x-x_a(t)|^2dxdv\\
&\quad\quad+ C\int_{\r^{2d}}f^D|v-v_a(t)|^2dxdv.
\end{align}

Finally we fix a suitable $\delta>0$ and combine inequalities (\ref{fzv}) and (\ref{fzx}), which leaves with
\begin{multline*}
\frac{d}{dt}\left(\int_{\r^{2d}}(t^{-1}f^D|x-x_a(t)|+f^D|v-v_a(t)|^2)dxdv\right)+ \frac{1}{2}\int_{\r^{2d}}t^{-2}f^D|x-x_a(t)|^2dxdv\\
\leq A(t)\int_{\r^2d} (t^{-1}f^D|x-x_a(t)|^2+f^D|v-v_a(t)|^2) dxdv,
\end{multline*}
which by Gronwall's lemma and the fact that $A(t) \sim t^{-1/2-\alpha}$  is integrable in a neighborhood of $t=0$ (restriction $\alpha \in (0,\frac{1}{2})$ is used here again) implies
\begin{align*}
\int_{\r^{2d}}(t^{-1}f^D|x-x_a(t)|^2+f^D|v-v_a(t)|^2)dxdv\equiv 0 \mbox{ \ on \ } [0,T^*].
\end{align*}
 Thus on $[0,T^*]$ we have $x\equiv x_a$ and $v\equiv v_a$ on the support of $f^D$, which is exactly equivalent to (\ref{fd}). 

We have proved $f^D$ is mono-atomic. Then repeating the procedure for all atoms (the number is finite) we conclude that $f$ is atomic
 on a time interval $[0,T^*]$ with possibly smaller, but positive $T^*>0$. This procedure works till the first moment of sticking of an ensemble of particles.

As a final remark we explain the case of the sticking some particles in a finite time, say $T_1$. The above considerations prove uniqueness and atomic structure of the solutions 
for the time interval $(0,T_1)$ without sticking of particles, and we want to reach $T_1$.
The regularity of the weak solution guarantees that $\partial_t f \in L^p(0,T;(C^1(B(R)))^*)$ with $T> T_1$, hence trajectories of atoms are $W^{1,p}$ vector functions and they are uniquely
extended up to $T_1$. Moreover, time regularity and (\ref{ugh}) exclude other possibilities of evolution of the studied measure-valued solution. Thus, we can reinitiate our analysis from time $T_1$ with 
atomic initial state $f(T_1)$, reaching given $T$. Note that the number of moments of sticking is finite, since $N$ is finite (see Remark \ref{N-N}).
 As a conclusion, since the solution exits globally in time it must be atomic on $[0,T]$. Theorem \ref{main3} is proved.
\end{proof}

\appendix
\section{Appendix}\label{app}
\begin{proof}[Proof of Proposition \ref{pop}]
The existence and uniqueness part as well as points $(i)$ and $(ii)$ are no different than in the case of regular weight and we will not prove them here. Their proofs can be found in the 
literature (see for instance \cite{haliu} or \cite{jpe}). Thus it remains to prove  $(iii)$-$(v)$.\\
{$(iii)-(v)$}\\
First, assuming for notational simplicity that $(x^n,v^n,N_n,m_i^n)=(x,v,N,m_i)$ let us prove a particularly useful estimate. Let $1<p<q$ be given numbers satisfying additional conditions 
that will be specified later. For each $n=1,2,...$, velocity $v^n$ (denoted by $v$) is absolutely continuous on $[0,T]$ and thus by $(\ref{cs})_2$, we have

\begin{align}
m_i\int_0^T|\dot{v}_i|^pdt
&=m_i\int_0^T\left|\sum_{j=1}^Nm_j(v_j-v_i)\psi_n(|x_i-x_j|)\right|^pdt\nonumber\\
&\leq \sum_{j=1}^Nm_im_j\int_0^T|v_j-v_i|^p\psi^p_n(|x_i-x_j|)dt\nonumber\\
&= \sum_{j=1}^N\int_0^T\left(m_im_j\right)^\frac{p}{q} |v_j-v_i|^{p\cdot\frac{p}{q}}\psi^p_n(|x_i-x_j|) \cdot \left(m_im_j|v_j-v_i|^{p}\right)^{(1-\frac{p}{q})}dt\nonumber\\
&\leq \sum_{j=1}^Nm_im_j\int_0^T|v_j-v_i|^p\psi_n^q(|x_i-x_j|) dt + \sum_{j=1}^Nm_im_j\int_0^T|v_j-v_i|^pdt\label{qneq1.apA}\\
&\leq \epsilon\sum_{j=1}^Nm_im_j\underbrace{\int_0^T|v_j-v_i|^2\psi_n^\frac{2q}{p}(|x_i-x_j|) dt}_{=:A} + C(\epsilon)Tm_i + \sum_{j=1}^Nm_im_j\int_0^T|v_j-v_i|^pdt\label{vdot.apA}.
\end{align}
Inequality (\ref{qneq1.apA}) is obtained by Young's inequality with exponent $\frac{q}{p}$ while (\ref{vdot.apA}) follows by Young's inequality with exponent $\frac{2}{p}$. In both of the above 
inequalities we also use the assumption that $\sum_{i=1}^Nm_i=1$.  

Furthermore recalling that $\psi_n^\frac{2q}{p}(s)\leq\psi^\frac{2q}{p}(s)=|s|^{-\lambda}$, where $\lambda:=\frac{2q\alpha}{p}$, integral $A$ can be estimated as follows:
\begin{align*}
A\leq \sum_{k=1}^d\int_0^T(v_j^k-v_i^k)\cdot(v_j^k-v_i^k) |x_i^k-x_j^k|^{-\lambda}dt
 =  \sum_{k=1}^d\int_0^T(v_j^k-v_i^k)\cdot\left((x_j^k-x_i^k) |x_i^k-x_j^k|^{-\lambda}\right)^{'}dt \\
= -\frac{1}{1-\lambda}\sum_{k=1}^d\int_0^T(\dot{v}_j^k-\dot{v}_i^k)\cdot (x_j^k-x_i^k) |x_i^k-x_j^k|^{-\lambda}dt
+ \frac{1}{1-\lambda}\sum_{k=1}^d(v_j^k-v_i^k)\cdot (x_j^k-x_i^k) |x_i^k-x_j^k|^{-\lambda}\bigg|^T_0\\
\leq C_\lambda\int_0^T|\dot{v}_i||x_i-x_j|^{1-\lambda}dt
+ C_\lambda\int_0^T|\dot{v}_j||x_i-x_j|^{1-\lambda} dt
+ 2C_\lambda\sup_{t\in[0,T]}|v_j-v_i||x_i-x_j|^{1-\lambda}.
\end{align*}
However, the above estimation is valid only if $\lambda<1$, which means that $\frac{q}{p}\cdot 2\alpha<1$ and such condition can be easily satisfied if $\alpha<\frac{1}{2}$ and $1<p<q$ 
are small enough. By point $(ii)$ we have $|v|\leq {\mathcal R}$ and $|x|\leq {\mathcal R}$. This leads to the concluding estimation of $A$, which reads:
\begin{align}\label{szacA.apA}
A\leq C ({\mathcal R})^{1-\lambda}\int_0^T|\dot{v}_i| dt +
 C({\mathcal R})^{1-\lambda}\int_0^T|\dot{v}_j|dt + C({\mathcal R})^{2-\lambda}.
\end{align}
Now we will apply the above calculation (particularly estimations (\ref{vdot.apA}) and (\ref{szacA.apA})) in the effort to prove $(iii)$ and $(iv)$. For $(iii)$ let us assume that
 $p=q=1$\footnote{Note that (\ref{qneq1.apA}) remains true also for $p=q=1$.}. We sum (\ref{vdot.apA}) over $i=1,...,N$ to get
\begin{align*}
\sum_{i=1}^Nm_i\int_0^T|\dot{v}_i|dt\leq \epsilon\sum_{i,j=1}^Nm_im_j A + C(\epsilon)T + 2{\mathcal R}T
\end{align*}
and plug in (\ref{szacA.apA}) to obtain
\begin{align*}
\sum_{i=1}^Nm_i\int_0^T|\dot{v}_i|dt \leq 2\epsilon C({\mathcal R})^{1-\lambda}\sum_{i=1}^Nm_i\int_0^T|\dot{v}_i|dt + \epsilon C({\mathcal R})^{2-\lambda} + C(\epsilon)T + 2{\mathcal R}T,
\end{align*}
which after fixing sufficiently small $\epsilon$ and rearranging yields
\begin{align}\label{p1.apA}
\sum_{i=1}^Nm_i\int_0^T|\dot{v}_i|dt \leq C({\mathcal R})^{2-\lambda} +CT+C{\mathcal R}T,
\end{align}
which proves $(iii)$ for $p=1$. Then for $1<p=q$ using (\ref{vdot.apA}), (\ref{szacA.apA}) and (\ref{p1.apA}), we have
\begin{align}\label{pneq1.apA}
\sum_{i=1}^Nm_i\int_0^T|\dot{v}_i|^pdt\leq 2C({\mathcal R})^{1-\lambda}\sum_{i=1}^Nm_i\int_0^T|\dot{v}_i|dt + C({\mathcal R})^{2-\lambda} + CT + C{\mathcal R}^p T
\leq C({\mathcal R},p,T,\lambda)
\end{align}
and $(iii)$ is proved for some sufficiently small $p>1$. In order to prove $(iv)$ we take $1=p<q$ in (\ref{vdot.apA}), which leads us to a very similar result to (\ref{pneq1.apA}) and to 
the end of the proof of $(iv)$.\\

Let us prove {$(v)$}. Fix $n=1,2,...$ and a bounded, Lipschitz continuous function $g=g(x,v)$. Then according to Definition \ref{weakdef}, for $t\in[0,T)$, $\epsilon>0$ and
\begin{align*}
\chi_{\epsilon,t}(s):=
\left\{
\begin{array}{ccc}
1& {\rm for} & 0\leq s\leq t-\epsilon\\
-\frac{1}{2\epsilon}(s-t-\epsilon) & {\rm for} & t-\epsilon<s\leq t+\epsilon\\
0 & {\rm for} & t_\epsilon<s
\end{array}
\right.
\end{align*}
the function $\phi(s,x,v):=\chi_{\epsilon,t}(s)g(x,v)\in{\mathcal G}$ is a good test function in the weak formulation for each $f_n$. Thus we plug $\phi$ into (\ref{weak}) obtaining
\begin{align*}
\frac{1}{2\epsilon}\int_{t-\epsilon}^{t+\epsilon}\int_{\r^{2d}}f_ngdxdvdt = \\
=-\int_0^T\int_{\r^{2d}}f_n\chi_{\epsilon,t}v\nabla gdxdvdt - \int_0^T\int_{\r^{2d}}F_n(f_n)f_n\chi_{\epsilon,t}\nabla_v
 g dxdvdt  -\int_{\r^{2d}}f_{0}gdxdv.
\end{align*}
Since $t\mapsto\int_{\r^{2d}}f_ngdxdv$, $t\mapsto \int_{\r^{2d}}f_n\chi_{\epsilon,t}v\nabla gdxdv$ and $t\mapsto \int_{\r^{2d}}F_n(f_n)f_n\chi_{\epsilon,t}\nabla_v g dxdv$ are 
integrable functions (for fixed $n$ and $g$), then converging with $\epsilon\to 0$ leads to the following equation holding for a.a $t\in[0,T)$:
\begin{align*}
\int_{\r^{2d}}f_n(t)gdxdvdt &= \int_0^t\int_{\r^{2d}}f_nv\nabla gdxdvdt + \int_0^t\int_{\r^{2d}}F_n(f_n)f_n\nabla_v g dxdvdt - \int_{\r^{2d}}f_{0}gdxdv\\
 &= \int_0^t G(t)dt - \int_{\r^{2d}}f_{0}gdxdv,
\end{align*}
where
\begin{align*}
G(t)&:=\int_{\r^{2d}}f_n(t)v\nabla gdxdv + \int_{\r^{2d}}F_n(f_n)(t)f_n(t)\nabla_v g dxdv\\
&=\sum_{i=1}^Nm_iv_i^n(t)\nabla g(x_i^n(t),v^n_i(t)) + \sum_{i,j=1}^Nm_im_j(v_j^n(t)- v_i^n(t))\psi(|x^n_i(t)-x^n_j(t)|)\nabla_vg(x^n_i(t),v^n_i(t)).
\end{align*}
By virtue of points $(ii)$ and $(iii)$ of this proposition, we have
\begin{align*}
\int_0^T|G(t)|^pdt &\leq \int_0^T\left|\sum_{i=1}^{N}m_iv^n_i(t)(\nabla g)(x^n_i(t),v^n_i(t))\right|^pdt \\
&+ \int_0^T\left|\sum_{i,j=1}^{N}m_im_j\psi_n(|x_i^n(t)-x_j^n(t)|)(v^n_j(t)-v^n_i(t))(\nabla_v g)(x_i^n(t),v_i^n(t))\right|^pdt\\
&\leq Lip(g)^pT({\mathcal R})^p + Lip(g)^pM({\mathcal R}) =: M_g(Lip(g),{\mathcal R})
\end{align*}
which finishes the proof of $(v)$.
\end{proof}

Now we aim at giving a sketch of the proof of uniqueness to the particle system (\ref{cs}). Note that point (ii) of Proposition \ref{pop} implies that solutions to 
ODEs (\ref{cs}) are of $W^{2,p}$ regularity for
some $p>1$.

\begin{proof}[Proof of the uniqueness part of Theorem \ref{jpsmain}]

Consider two solutions to the system (\ref{cs}) with the same initial data, name them $(x_i,v_i)$  and $(\bar x_i,\bar v_i)$. They fulfill the following systems
\begin{equation}
\begin{array}{ll}
\displaystyle \frac{d}{dt}x_i=v_i, & \displaystyle \frac{d}{dt}\bar x_i=\bar v_i,\\[7pt]
\displaystyle \frac{d}{dt}v_i=\sum_{j} m_j(v_j-v_i)\psi(|x_i-x_j|), & \displaystyle \frac{d}{dt}\bar v_i=\sum_{j} m_j(\bar v_j-\bar v_i)\psi(|\bar x_i-\bar x_j|).
\end{array}
\end{equation}
Putting
\begin{equation}
\delta x_i=x_i-\bar x_i, \qquad \delta v_i = v_i -\bar v_i,
\end{equation}
we find
\begin{equation}\label{xx2}
\begin{array}{l}
\displaystyle \frac{d}{dt}\delta x_i=\delta v_i, \\[7pt]
\displaystyle \frac{d}{dt}\delta v_i=\sum_{j} m_j(\delta v_j-\delta v_i)\psi(|\bar x_i-\bar x_j|) - \sum_{j} m_j( v_j- v_i)(\psi(|\bar x_i-\bar x_j|)-\psi(|x_i-x_j|)).
\end{array}
\end{equation}

Now we use fine properties of solutions to (\ref{cs}). The solutions to (\ref{cs}) are atomic, so in particular they fulfill Definition \ref{weakdef}. The most important features
concern propagation of the support. In the language of the ODE solutions, they control the change of trajectories $x_i(t)$ and $v_i(t)$. 
In particular we have
$$
|\bar x_i (t)t- \bar x_j(t)| \geq \min\{C,Ct\}.
$$
There are two options: $x_{0,i}\neq x_{0,j}$, then the difference is bounded from below by a constant (on a short time interval), or 
$x_{0,i}=x_{0,j}$ and $v_{0,i}\neq v_{0,j}$. Then particles travel in different cones and by Lemmas (\ref{excon}) and (\ref{conelem})  we find that the difference is bounded by $t$.
Hence at least for a short time 
$$
\psi(|\bar x_i(t)-\bar x_j(t)|) \leq C+Ct^{-\alpha}
$$
Next we examine
$\psi(|\bar x_i-\bar x_j|)-\psi(|x_i-x_j|)$.
Here we have again two  cases. The first one holds as $x_{0,i}\neq x_{0,j}$. Then for short time interval 
$$
\big| \psi(|\bar x_i(t)-\bar x_j(t)|)-\psi(|x_i-x_j|) \big| \leq C( |\delta x_i(t)| + |\delta x_j(t)|).
$$
In the second case $x_{0,i}=x_{0,j}$ and $v_{0,i}\neq v_{0,j}$. Then by Lemma \ref{excon}
$$
{\rm supp \;} x_i(t),\bar x_i(t) \in x_{0,i} + tB(v_{0,i},\epsilon) \mbox{ \ and \ } {\rm supp\; } x_j(t),\bar x_j(t) \in x_{0,i} + tB(v_{0,j},\epsilon),
$$
then
$$
|x_i(t)-x_j(t)|,\, |\bar x_i(t)-\bar x_j(t)| \geq Ct
$$
for $t \in [0,T_*]$. Hence ($\psi'(|s|) \sim s^{-1-\alpha}$)
 $$
\big| \psi(|\bar x_i(t)-\bar x_j(t)|)-\psi(|x_i-x_j)| \big| \leq Ct^{-1-\alpha}( |\delta x_i(t)| + |\delta x_j(t)|).
$$
Next, trivially from (\ref{xx2}) we find
$$
\sup_{\tau \leq t}|\delta x_i(\tau)| \leq t\sup_{\tau \leq t} |\delta v(\tau)|.
$$
Hence we obtain
$$
|\frac{d}{dt} \delta v_i(t)| \leq \sum_j m_j  C t^{-\alpha} (\sup_{\tau \leq t}|\delta v_i(\tau)|+\sup_{\tau \leq t}|\delta v_j(\tau)|)
+\sum_j m_j Ct^{-1-\alpha} (\sup_{\tau \leq t}|\delta x_i(\tau)| +\sup_{\tau \leq t}|\delta x_j(\tau)|).
$$
And finally 
$$
\sum_i m_i\sup_{t\leq T} | \delta v_i(t)|
  \leq C \int_0^T (t^{-\alpha} + t^{-1-\alpha} t) dt \,  \left(\sum_i m_i \sup_{\tau \leq t}|\delta v_i(\tau)|  \right).
$$
Taking  $T\leq T_*$ so small that $C \int_0^T (t^{-\alpha} + t^{-1-\alpha} t) < \frac12$ we obtain
$$
\sum_i m_i\sup_{t\leq T} | \delta v_i(t)| \equiv 0.
$$
The solutions are unique.
The above procedure works till the time of the first sticking of a group of particles, say at time $T_1$. Then the regularity in time ($v\in W^{1,p}(0,T)$) implies that all solutions 
are uniquely extended till time $T_1$. And then we can restart our procedure with 
initial configuration $x_i(T_1),v_i(T_1)$, taking into account that sticking particles create a new one with an appropriate mass.

\end{proof}

{\bf Acknowledgements}
We would like to express our gratitude to Jos\' e A. Carrillo and Piotr Gwiazda for numerous helpful remarks.

\bibliographystyle{spmpsci}
\bibliography{CS-fin-rev}
\end{document}